\documentclass[11pt]{article}

\usepackage[a4paper,margin=1in]{geometry}
\usepackage[T1]{fontenc}
\usepackage[utf8]{inputenc}
\usepackage{lmodern}
\usepackage{microtype}

\usepackage{amsmath,amssymb,amsthm,mathtools}
\usepackage{mathrsfs}
\usepackage{bm}
\usepackage{enumitem}
\usepackage{xcolor}
\usepackage{hyperref}
\usepackage[nameinlink,capitalize]{cleveref}
\usepackage{tikz-cd}
\usepackage{graphicx}
\usepackage{subcaption}
\usepackage{float}

\hypersetup{
    colorlinks=true,
    linkcolor=blue!60!black,
    citecolor=blue!60!black,
    urlcolor=blue!60!black,
    pdftitle={Working Title},
    pdfauthor={Your Name},
    pdfsubject={Probability Theory},
    pdfkeywords={probability, stochastic processes, random variables}
}

\theoremstyle{plain}
\newtheorem{theorem}{Theorem}[section]
\newtheorem{proposition}[theorem]{Proposition}
\newtheorem{lemma}[theorem]{Lemma}
\newtheorem{corollary}[theorem]{Corollary}

\newtheorem{conjecture}[theorem]{Conjecture}

\theoremstyle{definition}
\newtheorem{definition}[theorem]{Definition}

\newtheorem{remark}[theorem]{Remark}

\allowdisplaybreaks

\title{Interference Queueing Networks: A Replica Mean-Field Approach in the Symmetric Setting}
\author{
    Philippe Sarotte\thanks{INRIA, Paris}, 
    Nahuel Soprano-Loto\footnotemark[1], 
    François Baccelli\footnotemark[1]
}
\date{\today}

\begin{document}

\maketitle

\begin{abstract}
We propose a model for evaluating the performance of wireless communication networks beyond the ubiquitous full-buffer assumption, under which every transmitter is always active. The network is represented by $N$  interacting queues arranged on a torus, with homogeneous arrival rate $\lambda$  and service rates depending on the activity of neighboring interferers. More precisely, each queue is associated with a transmitter–receiver pair, and its service rate is given by the Shannon capacity, which depends on the corresponding Signal-to-Interference-plus-Noise Ratio (SINR). Since interfering transmitters only emit when their queue is non-empty, the SINR -- and hence the service rate -- improves when neighboring queues are empty. We first derive the stability region of the system. To investigate the stationary regime, we introduce two mean-field approximations. The first one is obtained from a finite \(K\)-replica system, for which we prove propagation of chaos as \(K \to \infty\). The second one describes a queueing system evolving in a suitably chosen environment. We derive quantitative estimates comparing these two approximations. Finally, we prove that the original interacting queueing system converges exponentially fast to stationarity. 
\end{abstract}


\tableofcontents

\bigskip

\section{Introduction}

\subsection{Motivation}

The goal of this paper is to introduce a tractable model for analysing the performance of wireless networks operating over a shared bandwidth beyond the classical full-buffer assumption. Rather than assuming that all transmitters are perpetually backlogged, we consider a system with bursty traffic, where packets arrive dynamically and are stored in queues before transmission. The resulting interaction between traffic dynamics and interference leads to a coupled queueing system whose analysis is challenging. Beyond its theoretical interest, this finer-grained analysis can provide useful guidance for network dimensioning and engineering. By explicitly accounting for the interaction between packet arrivals, queue occupancy, and interference, it becomes possible to assess more accurately the admissible load, the density of simultaneously active links, and the conditions under which full spectrum sharing remains viable. This, in turn, can inform the design of current wireless networks by indicating when mechanisms such as admission control, power control, or partial spectrum partitioning are needed to avoid persistent congestion.

We model the network as a collection of transmitter--receiver pairs sharing the same wireless medium. Each transmitter stores packets  arriving at a rate $\lambda$ in a queue and serves them over its associated wireless link. When several neighbouring transmitters are simultaneously active, they interfere with one another, so that the transmission rate of each active link depends on the current activity of the surrounding queues. Under adaptive coding and ideal Gaussian signalling, this dependence is captured through the Shannon capacity formula, which links the service rate of each queue to the instantaneous signal-to-interference-plus-noise ratio. The service process is therefore state-dependent, and all queues become coupled through interference.
Our aim in this paper is to study the fundamental queue--interference coupling within a mathematically tractable framework, at the expense of several simplifying assumptions.
These assumptions are admittedly restrictive, but they allow us to introduce the main phenomena and develop mathematical tools that can later be extended to more realistic settings. We now summarise them.

\begin{itemize}
    \item \textbf{One queue per wireless link.} Each queue corresponds to a transmitter--receiver pair, with each transmitter serving at most one receiver at a time.
    
    \item \textbf{Short links.} Receiver--transmitter distances are assumed small compared with inter-transmitter distances, i.e., $T \ll d(i,j)$. Hence, the interference at a receiver is approximated by that at its transmitter.
    
    \item \textbf{Homogeneous queues.} All queues are statistically identical, with common arrival rate, file-size distribution, transmit power, link distance, and fading statistics.
    
    \item \textbf{Interference-dependent service.} A queue service rate depends on the set of simultaneously active transmitters through interference, according to the Shannon capacity formula.
    
    \item \textbf{Full frequency reuse.} All links share the same frequency band, so interference is not mitigated by orthogonal resource allocation.
    
    \item \textbf{Symmetric geometry.} Transmitters are equally spaced on a torus, ensuring identical interference conditions for all nodes and eliminating edge effects. Periodic boundary conditions should be understood as a proxy for a finite window within a larger homogeneous network.
    
    \item \textbf{Static channel model.} We consider a quasi-static channel with constant noise over the time interval of interest. Under translation invariance, links with the same relative geometry have identical fading statistics, and instantaneous fading is replaced by its mean power in the SINR.
    
    \item \textbf{Ideal rate adaptation.} Rate adaptation is assumed instantaneous and overhead-free, so transmission always operates at the Shannon capacity corresponding to the current interference state.
\end{itemize}

Within this setting, several natural performance questions arise. Can the stability region of the system be characterised explicitly? When the system is stable, can one describe its stationary regime and derive performance metrics such as mean delay or congestion? Can the analysis be extended to random spatial configurations? Finally, what can be said about the long-time behaviour of the system and its convergence to equilibrium?

\subsection{Related Work}

Stochastic analysis has long been central to the performance evaluation of telecommunication networks, providing mathematically grounded approximations and scaling laws that help identify key design parameters without costly deployment campaigns or extensive measurements; see, e.g., \cite{10.1561/1300000026}.

Several works have relaxed the full-buffer assumption in wireless settings. In \cite{sankararaman2017spatial}, a device-to-device model is introduced in which files are transmitted at the Shannon rate in a bounded continuous space with randomly located transmitters. A related model on an infinite countable space with deterministic transmitter locations is studied in \cite{sankararaman2019interference}. For cellular downlink systems, \cite{alaammouri2020stability} considers the single-cell case, while \cite{kaj2023retransmission}, \cite{yang2018sir}, \cite{bonaldWirelessData}, and \cite{Blaszczyszyn2014} study multi-cell models with randomly located users and explicit multiple-access mechanisms.

A major challenge in the analysis of such systems is the presence of spatial and temporal correlations between particles, which typically make an exact treatment intractable. A natural approach is therefore to approximate the original dynamics by a mean-field limit; see, e.g., \cite{kaj2023retransmission,mitzenmacher2001power}. While mean-field limits provide a powerful analytical framework, proving the existence of the limiting model may itself be delicate, as the limits in time and in the number of particles do not necessarily commute. This phenomenon appears, for instance, in the Kuramoto model \cite{BertiniGiacominPakdaman2010Kuramoto} and in the Curie--Weiss model \cite{JournelBris2025CurieWeiss}. The long-time behaviour of either the finite system or its mean-field approximation has been investigated in various settings, such as \cite{Graham2000}, where the limiting dynamics are governed by an ODE, and \cite{Cormier2020long}, where the McKean--Vlasov nonlinearity enters through the drift.

\subsection{Contributions}

We now summarize the main contributions of the paper, their interpretation in plain terms, the main technical obstacles encountered in the analysis, and their relevance for performance evaluation.

As described in the motivation section, the system consists of interacting queues receiving jobs at rate $\lambda$ and transmitting them at rates that depend on the activity of the whole network. More congested systems have smaller transmission rates. This creates a natural competition between the arrival mechanism, which tends to increase the workload, and the service mechanism, which allows the system to empty itself.

Our first contribution is to make this intuition rigorous. We prove the existence of a critical threshold for the arrival rate. Below this threshold, the system is stable, in the sense that it returns to the empty state infinitely often. Above it, with probability $1$ the number of jobs diverges to infinity. This threshold has a natural interpretation: the arrival rate must be smaller than the worst-case transmission rate $\mu^*$, corresponding to the rate at which a queue transmits when all surrounding queues are active. We also realize that it means that taking into account the empty slots does not enlarge the stability region of the queueing system. 

To establish this result, we first provide a stochastic differential equation representation of the system, in a form reminiscent of interacting neuron models. This formulation makes it possible to construct explicit couplings by using the same Poisson random measures. In the regime $\lambda < \mu^*$, stability follows from stochastic domination by the full-buffer system, whose explosion region is known. The converse direction is more delicate, since a direct stochastic domination argument is no longer available. Instead of relying on Lyapunov methods or fluid-limit techniques at this stage, we use coupling arguments that exploit the fact that the true system behaves like the full-buffer system whenever no queue is empty.

We then turn to the stable regime and to the problem of estimating stationary performance metrics, such as delay, congestion, and busy probability. Obtaining closed-form expressions, or even sampling from the exact stationary distribution, is a notoriously difficult problem. In the present setting, the difficulty is closely related to the analysis of random walks in the positive orthant with boundary interactions. Existing results, in particular those of~\cite{fayolle2017random}, indicate that an exact description is in general out of reach. We therefore adopt an approximation approach.

The purpose of our approximations is to remove the main source of analytical difficulty, namely the randomness of the transmission rates induced by the current state of the other queues. Mean-field approximations achieve this by replacing this random quantity with a simpler averaged one. This reduction leads to tractable objects, essentially of $M/M/1$ type, for which stationary performance metrics can be computed. A further difficulty, however, is that the approximation should still preserve the geometry of the original system, in particular the notion of distance between queues. To this end, we introduce a replica-type mean-field construction.

More precisely, we define a replica-type prelimit system and prove propagation of chaos towards a limiting system in which the random interaction term is replaced by its averaged counterpart. This gives a rigorous interpretation of the approximation as the limit of a finite interacting system. The proof relies on classical Sznitman coupling techniques, adapted to Poisson random measures, in the spirit of mean-field limits for interacting neuron systems. In this way, the difficulty associated with random service rates in the original system is transferred to the McKean--Vlasov nonlinearity of the limiting system. Nevertheless, in stationarity, the limiting queues recover an $M/M/1$-type structure, which makes the approximation analytically tractable. We also observe that the well-posedness region of the approximation coincides with the stability region of the true system. This agreement further supports the relevance of the approximation.

We also introduce a second mean-field approximation. In this case, we do not address the question of identifying a prelimit system from which it arises. Instead, we focus on replacing the random transmission rate based on the current state of the system by the additional independent mark we add to the Poisson driver. This leads to another $M/M/1$ system. Using functional inequalities and the convexity properties of the Shannon rate, we obtain qualitative comparisons between the two mean-field approximations through Jensen-type arguments.

The analytical results are then compared with numerical simulations. The simulations confirm the theoretical predictions and highlight a significant difference between the full-buffer system and the true system. This confirms the need for a refined analysis that accounts for empty queues.  Moreover, as the impact of interference decreases, the discrepancies between the different systems become smaller. One can check the simulation at the \href{https://github.com/Philippe-Sarotte/Interference-Queueing-Networks-A-Replica-Mean-Field-Approach-in-the-Symmetric-Setting}{associated GitHub repository} (https://github.com/Philippe-Sarotte/Interference-Queueing-Networks-A-Replica-Mean-Field-Approach-in-the-Symmetric-Setting).

We further study the long-time behaviour of the true system. To do this, we prove a Foster--Lyapunov drift condition and then use exponential ergodicity criterion. This result is relevant for numerical implementation, since it justifies that stationary samples can be obtained without simulating the system over excessively long time horizons.

Finally, we extend the analysis to systems with randomly spaced queues. This extension emphasizes that our approximations preserve the underlying geometry of the network and that the mean-field results can still be applied in this more general spatial setting.

\subsection{Structure of the paper}

The remainder of the paper is organized as follows. \Cref{sec:setting} introduces the main physical concepts underlying the analysis, fixes the notation, and presents the model in two equivalent formulations: through its Markovian description and through its SDE representation. The latter point of view will play an important role throughout the paper.  
\Cref{sec:main_r} states the main results, first in a mathematically rigorous form and then in a more intuitive formulation. More precisely, \Cref{sec:st_reg} describes the stability region, while \Cref{sec:stat_dist} studies the stationary distribution and its approximation by two McKean--Vlasov systems. One of these systems arises as the limit of a system of replicas, for which propagation of chaos is established. Next, \Cref{sec:long_time} investigates the long-time behavior of the original system. \Cref{sec:rd_pos} then considers random transmitter positions, which, in the symmetric case, amounts to random common spacings.  
The proofs of the results are given in \Cref{sec:proof}, in the same order as their presentation. \Cref{sec:extensions} discusses several extensions that will be addressed in subsequent papers. Finally, \Cref{sec:appendix} collects auxiliary results and proofs, included either for completeness or to avoid interrupting the flow of the main text.

\section{Setting, notation, and assumptions}\label{sec:setting}
\subsection{Assumption on the Physical Model}
As mentioned in the introduction, our model applies to almost any network sharing a common bandwidth, in which each transmitter serves only one receiver at a time and the distances between transmitters are relatively large compared with those between transmitters and receivers. As already acknowledged, one may argue that this framework is not fully realistic, especially given the additional assumptions of homogeneous arrival rates, equally spaced transmitters, assuming the availability of capacity-achieving codebooks across different channel conditions and density of the interaction graph neglecting frequency reuse. However, we emphasize that this paper is the first in a sequence of works whose purpose is to progressively relax these assumptions in order to better reflect reality. The goal here is to introduce the main ideas and mathematical tools in the simplest possible setting, so that the exposition remains clear.
\subsection{Key Physical quantities }
 Here we describe precisely the main quantities of the model, namely the Signal to interference Noise ratio (SINR) and the Shannon rate. 
\subsubsection{SINR}
The Signal to Interference Noise ratio is a key quantity to understand whether signal can be transmitted reliably at a given rate. For a scenario where a transmitter located at position \(0\) sends a message to a receiver at position \(x\) in the presence of other transmitters, having a message to send, located in the atomic set \(S\), we will use the following definition: 
\begin{align*}
    \mathrm{SINR}(x)=\frac{P_e\, h_0\, l(0,x)}{W+\sum_{y \in S} P_e\, h_y\, l(y,x)}.
\end{align*}
\begin{enumerate}
    \item  The transmissions power is \(P_e\). In this framework, all transmitters are assumed to operate at the same constant power.
    \item The fading coefficient is denoted by \(h_y\) for the \(y\)-th transmitter. This term models channel fluctuations due to multi path in a simplified manner without solving the full electromagnetic wave equations. Typically, \(h_y\) is a random variable following a specific distribution, for example, Rayleigh, see \cite{10.1561/1300000026}. Here we assume it is constant.
    
    \item  The path loss function is \(l(x,y)\). It accounts for the reduction in signal power with distance. The further away  the receiver is from the transmitter, the weaker the received signal. Common models include power-law functions, which may or may not exhibit a singularity at \(0\). In this paper, we consider a family of path-loss functions belonging to the following set:
\[
\mathcal{L}(A,C,r_0) \;=\; \left\{
\begin{aligned}
&\ell_1(r) = \frac{1}{\bigl(A \cdot \max(r,r_0)\bigr)^C}, \\[0.6em]
&\ell_2(r) = \frac{1}{\bigl(1 + A r\bigr)^C}, \\[0.6em]
&\ell_3(r) = A \, e^{-Cr}
\end{aligned}
\;\middle|\; A>0,\; C>0,\; r_0>0
\right\}.
\]

    \item \(W\) represents the power of the noise aside from interference, such as thermal noise. It can be modeled by a non negative random variable independent of the rest of the system. Here we assume it is constant.
\end{enumerate}

For the sake of readability, we henceforth set all constants equal to one. This does entail a loss of generality at the level of the model, but the analysis extends to arbitrary constants with only notational modifications.
\subsubsection{Shannon Rate}

The Shannon rate is another fundamental concept in telecommunications and information theory. When transmitting a message, one encodes it into symbols which are transmitted over a channel subject to perturbations before being decoded at the receiver. The goal is to design an encoding and decoding scheme such that the original message is recovered with probability arbitrarily close to one. In a basic, non-adaptive transmission scheme, successful transmission occurs if the SINR exceeds a certain threshold. However, when one has multiple codes that can be dynamically adapted to changing channel conditions, \cite{shannon1948mathematical} showed that when treating interference as noise, the maximum achievable rate of reliable information transmission is given by the Shannon capacity:
\begin{align*}
    C = B \log_2\bigl(1+\mathrm{SINR}\bigr),
\end{align*}
where \( B \) denotes the bandwidth used by the network. In our framework, we assume that we have coding schemes capable of matching any channel condition. We further assume that adaptation to channel conditions is instantaneous, and therefore ignore any delay induced by the adaptation mechanism.

\subsection{Description of the system and notation}\label{sec:descr}
\subsubsection{Markovian Description}
We consider a system of $N$ queues equally spaced on the torus $[0, N-1]$, 
where each job represents a file to be transmitted (see \Cref{fig:Figure 1}). The state of the system is 
described by the congestion process
\begin{align} \label{eq:state}
    Q(t) = \bigl(q_1(t), \dots, q_N(t)\bigr),
\end{align}
where $q_i(t)$ denotes the number of jobs present in queue $i$ at time $t$.

\begin{figure}[H]
    \centering
    \includegraphics[width=0.5\linewidth]{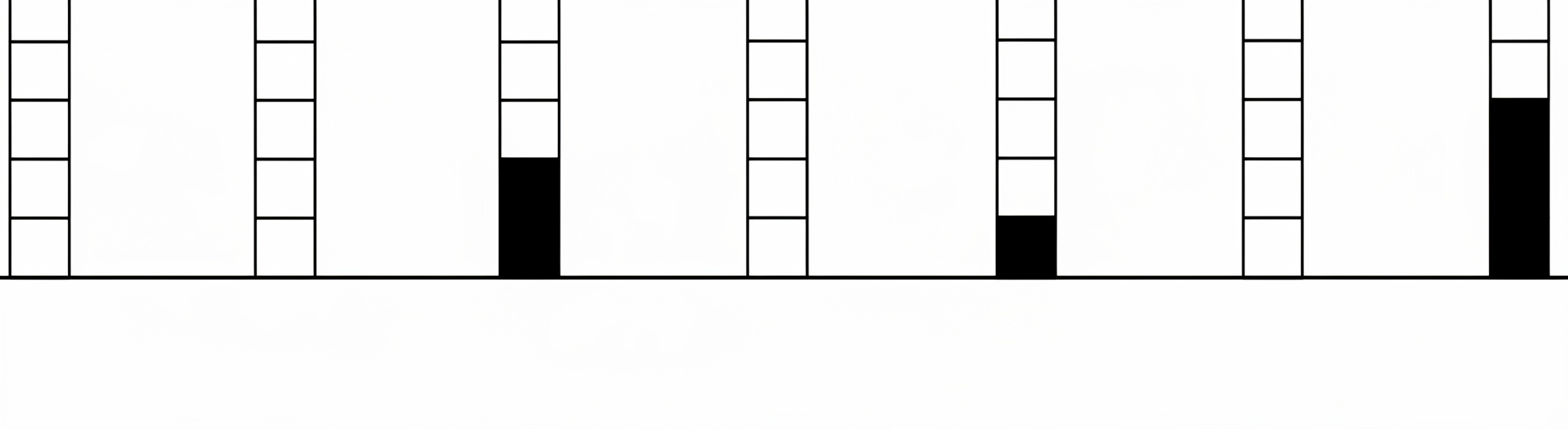}
    \caption{Mathematical representation of the symetric wireless network}
    \label{fig:Figure 1}
\end{figure}

We place ourselves in a Markov framework, hence $Q$ is a (non explosive) continuous time Markov chain with respect to its natural filtration and can be described thanks to its generator:  For any bounded $f$ from $\mathbb{R}^N$ to $\mathbb{R}$:
\begin{align}\label{eq:generator}
    \mathcal{L}f(q)
= \lambda \sum_{i} \big( f(q + e_i) - f(q) \big)
  \;-\;
  \sum_{i} \mu_i(q)\,\big( f(q) - f(q - e_i) \big)\,\mathbf{1}_{\{q_i > 0\}} ,
\end{align}

where \begin{itemize}
    \item $
\mu_i(q) = \frac{B}{L} \log_2\!\left(1 + \frac{1}{1 + \sum_{j \neq i} \mathbf{1}_{\{q_j > 0\}} \, l(d(i, j))}\right).
$
    \item  The vector $e_i$ is only 0 except at position $i$ there is a $1$.
    \item $l \in \mathcal{L}(A,C,r_0)$, 
    \item The circular distance between the i-th buffer and the j-th one is denoted by $d(i,j)$. 
    \item We define $\xi_N:=\sum_{j \neq i} l(d(i, j))$, which refers to the worst interference conditions, that is, all neighboring queues are non-empty. Note that it is a common quantity for each queues in our equally spaced, periodic boundary conditions. 
    \item In the same fashion we define $\mu^*:= \frac{B}{L} \log_2\!\left(1 + \frac{1}{1 + \xi_N}\right)$ and $\mu^{**}:= \frac{B}{L}$ as being the worst and best service rates.
    \item We assume initial condition are independent from the whole system, translation invariant and with finite first and second moment when it is needed. 
\end{itemize}

In essence, the system operates as follows:
\begin{itemize}
    \item A job is a file made of bits
    \item Job arrivals at each queue form independent Poisson processes with rate \( \lambda \).
    \item Job sizes are i.i.d.\ exponential random variables with parameter (mean) \( L \), independent of the arrival processes. We will take $L=1$ to simplify notations.
    \item The instantaneous job departure rate of queue \( i \) at time \( t \) is denoted by \( \mu_i(q(t)) \).
\end{itemize}

\begin{remark}
   The full-buffer system is obtained by replacing each \( \mathbf{1}_{\{q_j > 0\}} \) with \( 1 \). In this case, the system reduces to \( N \) independent \( M/M/1 \) queues. 
\end{remark}
\begin{remark}
    We used the Shannon rate for the departure rate but actually we can broaden the analysis to a more general class of function $g :[0;\xi_N] \to \mathbb{R}^+$, being bounded, non increasing, convex, smooth enough (mainly Lipschitz) and $\mu_i(t)=\frac{1}{L}g(\sum_{j \neq i} \mathbf{1}_{\{q_j(t) > 0\}} \, l(d(i, j)))$. We need an extra assumption about the sign around the fixed point function to conclude for \Cref{prop:mf2-mf1}. We do not go into further detail since there is a mathematical version of this paper under revision working under the most general assumption. 
    The Shannon rate description is retrieved by taking $g(x)=B\log_2(1+\frac{1}{1+x})$.
\end{remark}

\begin{remark}
    Thanks to the periodic boundary condition and the equal spacing of the queues, they have the same marginals. 
\end{remark}

\subsubsection{Dynamical description}
Here we describe the system in terms of SDE with respect to Poisson random measures which is, as we mentioned, 
crucial in our analysis . We will see that those two representations are equivalent in law thanks to Martingale problem considerations. 
Consider  $(A_i)_{i\in [0;N-1]}$, $N$ i.i.d. Poisson point processes of intensity $\lambda\rm{d}x$. In addition consider $(D_i)_{i \in [1;N]}$ i.i.d. Poisson point processes of intensity $\rm{d}xdy$. Now define $\Tilde{Q}=(\Tilde{q}_1, \dots,\Tilde{q}_N)$ as 
\begin{align}\label{eq:sde}
\widetilde q_i(t)
&=
\widetilde q_i(0)
+ A_i[0,t]
- \int_0^t \int_{\mathbb R_+}
\mathbf 1_{\{\widetilde q_i(z^-)>0\}}
\mathbf 1_{\{u \leq \mu_i(\widetilde Q(z^-))\}}
\, D_i(\mathrm du,\mathrm dz),
\qquad i=1,\dots,N .
\end{align}

\begin{remark}

   Equation~\eqref{eq:sde} is well posed and defines a process with generator~\eqref{eq:generator}; see \Cref{prop:sde-generator}. 
\end{remark}
\subsubsection{Notation}

\begin{itemize}
    \item For $x, y \in \mathbb{N}^N$, we denote by $x \leq y$ the partial order on $\mathbb{N}^N$ defined component-wise, that is,
    \[
    x \leq y
    \quad \text{if and only if} \quad
    x_i \leq y_i \text{ for all } i \in \{1, \dots, N\}.
    \]

    \item For two $\mathbb{N}^N$-valued random variables $X$ and $Y$, we write $X \leq_{\mathrm{st}} Y$ whenever $X$ is stochastically dominated by $Y$.

\end{itemize}

\section{Main Results}\label{sec:main_r}

\subsection{Stability Region}\label{sec:st_reg}

In this section, we state the stability region of the system.
Since $Q$ is a countable state space Markov chain, it is natural to characterize its region of positive recurrence. Our goal is to study this region as a function of the parameter $\lambda$. Intuitively, as the arrival rate increases, the number of jobs in the system also increases, leading to more severe interference conditions and, consequently, to congestion buildup. One may then ask whether there exists a threshold beyond which the system becomes unstable. We show that this is indeed the case. Relying on monotonicity properties and using the full-buffer system as a benchmark, we prove that this threshold is exactly $\mu^*$. In other words, stability requires the arrival rate to be smaller than the minimal service rate.

\begin{theorem}[Stability Region]\label{thm:stability}

For $\lambda > \mu^*$, the process $Q$ is transient, whereas for $\lambda < \mu^*$, it is positive recurrent.
\end{theorem}

\begin{remark}
    In \cite{fayolle2017random} necessary and  sufficient conditions for ergodicity (here
same as positive recurrence) were established for such chains in the case N = 2. One can check our condition match their. In  $N=2$ , they are also able to settle the boundary case  $\lambda=\mu^*$ , showing that the process is not positive recurrent — a result that goes beyond our analysis.
\end{remark}

\begin{remark}
    Analysing the system without the full buffer assumption does not change the performance compared to the full buffer system in terms of range of the stability region. 
\end{remark}

\begin{remark}
    Since the chain $Q$ is irreducible, for $\lambda < \mu^*$, the system admits a unique stationary distribution that we denote by $\pi_\lambda$
\end{remark}

We introduce here the notion of association and prove that the system satisfies this property. Informally, association means that the occupancy states of the queues are positively correlated: if some queues are busy, then the remaining queues are more likely to be busy as well, and similarly for emptiness. This is a notable structural property for interacting particle systems, and it plays a fundamental role in classical models such as the Ising model; see \cite{FriedliVelenik2017}. We provide afterwards a small example of application. 

\begin{definition}
    We say that $\mu$ is associated if for any $f$ and $g$ increasing real function for the partial order on $\mathbb{N}^N$,
    \begin{align}
        \mu(fg)\geq \mu(f)\mu(g)
    \end{align}
\end{definition}

\begin{definition}
    We say that $Q$ is associated if for any $f$ and $g$ increasing real function for the partial order on $\mathbb{N}^N$, and $t \in \mathbb{R}^+ $
    \begin{align}
        &\mathbb{P}^{Q(t)}(fg)\geq \mathbb{P}^{Q(t)}(f)\mathbb{P}^{Q(t)}(g) \Leftrightarrow \mathbb{E}[f(Q(t))g(Q(t))]\geq\mathbb{E}[f(Q(t))]\mathbb{E}[g(Q(t))]
    \end{align}
\end{definition}

\begin{proposition}[The system is associated] \label{prop:association}
    $\forall x \in \mathbb{N}^N, Q$ with initial law $\delta_x$ is associated.
\end{proposition}

\begin{corollary}[Tail Bound for the Maximum Queue Length] \label{cor:largest-queue-tail}
 Assume $Q(0)=(x, \dots, x)$ for some $x \in \mathbb{N}$. Define
\[
    M(t) := \max_{i\in[1:N]} q_i(t).
\]
Then, for every \(k\in\mathbb{N}\),
\begin{equation}
\label{eq:largest-queue-association}
    \mathbb{P}\bigl(M(t)\geq k\bigr)
    \leq
    1-
    \left(
        1-\mathbb{P}\bigl(q_1(t)\geq k\bigr)
    \right)^N.
\end{equation}

Moreover, for every \(t>0\), there exist constants \(C_t<\infty\) and
\(\kappa_t>0\) such that
\begin{equation}
\label{eq:largest-queue-exponential-tail}
    \mathbb{P}\bigl(M(t)\geq k\bigr)
    \leq
    1-
    \left(
        1-\min\{1,C_t e^{-\kappa_t k}\}
    \right)^N,
    \qquad k\in\mathbb{N}.
\end{equation}

\end{corollary}

\begin{remark}
\label{rem:association-vs-union-bound}
The association bound improves upon the union bound. Indeed, writing
$
    p_k(t):=\mathbb{P}\bigl(q_1(t)\geq k\bigr),
$
the union bound gives
$
    \mathbb{P}\bigl(M(t)\geq k\bigr)\leq Np_k(t),
$
whereas association gives
\[
    \mathbb{P}\bigl(M(t)\geq k\bigr)
    \leq
    1-\bigl(1-p_k(t)\bigr)^N
    \leq
    Np_k(t).
\]
\end{remark}

\subsection{Stationary distribution}\label{sec:stat_dist}

Now that the stability region has been identified, a natural question is to understand the stationary behavior of the system. This would require, in principle, a description of its stationary distribution. Such a description is difficult to obtain, since the problem can be viewed as that of finding the stationary distribution of a random walk in the positive orthant with boundary conditions. In dimension two, a direct characterization was obtained in \cite{fayolle2017random}. In higher dimension, however, the problem is known to become substantially more involved; see, for instance, \cite{Cohen1984}.

\subsubsection{Mean-Field 1}

Having characterized the maximal stability region, we now turn to the equilibrium behavior of the system. In principle, performance metrics could be derived along the lines of \cite{sankararaman2017spatial,sankararaman2019interference}, using the Rate Conservation Principle applied to suitable test functionals. In the present work, we instead take a mean-field approximation viewpoint.

The main difficulty comes from the state-dependent departure rates. It is therefore natural to compare the original system with a simplified one in which the relevant random quantity is replaced by its mean. 

An additional constraint is that the notion of distance between queues must be retained in the approximation. This prevents us from considering prelimit systems in which the number of queues tends to infinity, as in \cite{fournier2015toymodelinteractingneurons,Goldsztajn2026}. We therefore introduce a replica-type prelimit system, in the spirit of \cite{Baccelli2019}: it replaces the problematic random quantity by its limiting counterpart while preserving the underlying geometry.

\begin{definition}(K-Replica System) \label{def:replica}
    We denote by $Q^K=(q^{(K,l)}_i)_{l \in [1,K], i \in [1,N]}$ the K-Replica System which is a Markov chain on $(\mathbb{N}^N)^K$ defined by the generator : \begin{align} \label{eq:replica-generator}
        \mathcal{L}^Kf(q)
=
\lambda \!\sum_{i,l}\!\big(f(q + e_{(l,i)}) - f(q)\big)
-
\sum_{i,l}\!\mu^K_i(q)\big(f(q) - f(q - e_{(l,i)})\big)
\mathbf{1}_{\{q_{(l,i)}>0\}} ,
    \end{align}
where \begin{itemize}
    \item $e_{(l,i)}$ is the vector of $(\mathbb{N}^N)^K$ having a 1 position $l \times i $,
    \item $ \mu^K_i(t) = B \log_2\!\Biggl(1+\frac{1}{\,1+\frac{1}{K}\sum\limits_{k=1}^{K} I_i^k(t)}\Biggr)$,
    \item $I_i^k(t) = \sum_{j \neq i} \mathbf{1}_{\{q_j^{(K,k)}(t) > 0\}} \, l(d(i,j))$.
\end{itemize}
    
\end{definition}

\begin{remark}
    Note that every i-th queue of each replica experiences the same departure rate $\mu^K_i(t)$ for all times. 
\end{remark}

\begin{definition}[McKean-Vlasov 1]\label{def:mf1}
    Consider $A_{l,i}$ and $D_{l,i}$  Poisson processes of intensity $\lambda\rm{dx}$ and Poisson processes of intensity $\rm{dxdy}$ respectively, each being independent in $l,i$. The initial condition
\[
\left(q_i^{(MF1,l)}(0)\right)_{l\geq 1,\;0\leq i\leq N-1}
\]
is assumed to be invariant in law under translations of the index $i$ on the
discrete torus and identically distributed across the replicas. They are independent from the drivers. 
Define $Q^{MF1}=(q^{(MF1,l)}_i)_{l \in \mathbb{N}_*, i \in [1,N]}$
    \begin{align}\label{eq:MF1}
q_i^{(MF1,l)}(t)
&=
q_i^{(MF1,l)}(0)
+
A_{l,i}[0,t]
\nonumber \\
&\quad
-
\int_0^t \int_{\mathbb R_+}
\mathbf 1_{\{q_i^{(MF1,l)}(z^-)>0\}}
\mathbf 1_{\{u\leq \mu^{MF1}(z^-)\}}
\,D_{l,i}(\mathrm du,\mathrm dz),
\end{align}
where  $\mu^{MF1}(z)=B\log_2(1+\frac{1}{1+\mathbb{P}(q_1^{(MF1,1)}(z)>0)\xi_N})$.

\end{definition}

\begin{remark}
    Note that $\mu^{MF1}(z)$ does not depend on $i$, which is the main benefit of working in the symmetric setting. 
\end{remark}

Since \eqref{eq:MF1} is nonlinear, we first need to justify that the equation is well posed

\begin{theorem}[Well-posedeness of McKean--Vlasov]\label{thm:mf1-wellposed}

\eqref{eq:MF1} admits a strong solution which is path-wise unique.
\end{theorem}

In some interacting particle models, such a replacement, between the random quantity and the mean, arises directly in a suitable asymptotic regime of the original system; see, for example, \cite{fournier2015toymodelinteractingneurons}. In the present setting, however, this is not the case: one needs to introduce an auxiliary prelimit system for which the quantity of interest emerges in the appropriate limit.

Consequently, unlike in a direct limiting procedure, one should not expect immediate quantitative or qualitative bounds between the original system and the limiting system. Nevertheless, the two systems can be compared numerically, and the convergence of the auxiliary prelimit system toward the intended limit can be established rigorously.

\begin{theorem}[Propagation of chaos 1]\label{thm:chaos-one}
Let $T > 0$ and $i \in [1,N]$. Assume that, for all $K$, for all $l \leq K$, and for all $j \in [1,N]$,
\[
q_j^{(K,l)}(0) = q_j^{(MF1,l)}(0),
\]
and that $\bigl(q_j^{(MF1,l)}(0)\bigr)_{j,l}$ are i.i.d. with finite mean. Then for all $i \in [1;N]$,
\[
\bigl(q_i^{(K,1)}(t)\bigr)_{0 \leq t \leq T}
\xrightarrow[K \to \infty]{d}
\bigl(q_i^{(MF1,1)}(t)\bigr)_{0 \leq t \leq T}.
\]
\end{theorem}

\begin{theorem}[Propagation of chaos 2]\label{thm:chaos-joint}
Let $T > 0$. Assume that, for all $K$, for all $l \leq K$, and for all $j \in [1,N]$,
\[
q_j^{(K,l)}(0) = q_j^{(MF1,l)}(0),
\]
and that $\bigl(q_j^{(MF1,l)}(0)\bigr)_{j,l}$ are i.i.d. with finite mean. Then
\[
\bigl((q_i^{(K,1)}(t))_{i \in [1,N]}\bigr)_{0 \leq t \leq T}
\xrightarrow[K \to \infty]{d}
\bigotimes_{i=1}^{N}
\bigl(q_i^{(MF1,1)}(t)\bigr)_{0 \leq t \leq T}.
\]
\end{theorem}

The two last theorems show that, for large enough number replicas, if one looks at the trajectory of the first replica for a finite horizon of time $T$, it looks likes $N$ independent trajectories of particles evolving according \eqref{eq:MF1}. 

\begin{lemma} \label{lem:mf1-fixedpoint}
For $\lambda \in (0,\mu^*)$, there exists a unique solution in $[0,1]$ to
\begin{equation} \label{eq:mf1-fixedpoint}
(E): \qquad
\frac{\lambda}{x}
=
B \log_2\!\Bigl(1+\frac{1}{1+x\,\xi^N}\Bigr).
\end{equation}
\end{lemma}

\begin{theorem}[Existence and uniqueness of the stationary distribution of the McKean--Vlasov process] \label{thm:mf1-stationary}
Assume that $\lambda < \mu^*$. Then $Q^{(MF1,1)}$ admits a unique stationary distribution, which is given by the product measure with marginals
\[
\pi^{MF1}_\lambda \sim \operatorname{Geom}\bigl(1 - x^*(\lambda)\bigr),
\]
where $x^*(\lambda)$ is the solution of~\eqref{eq:mf1-fixedpoint}.
\end{theorem}

This illustrates the usefulness of the McKean–Vlasov approximation: stationary performance metrics can be obtained by solving the scalar fixed-point equation \eqref{eq:mf1-fixedpoint}.

\begin{remark}
    Observe that the Mean-field limit arising here is of a different nature from the one usually used to approximate such systems: the limit remains stochastic, instead of being deterministic and governed by an ODE as in \cite{Bordenave2009}, \cite{mitzenmacher2001power}, or \cite{10.1145/3154491}. This is a characteristic feature of replica mean-field models.
\end{remark}

\subsubsection{Mean field 2}

Having derived a first mean-field approximation in order to obtain a tractable description of the original coupled system, we now turn to an alternative approach. Instead of computing the interference from the actual activity states of the other queues, we replace these states by independent Bernoulli marks with a common parameter p. 

To this end, we introduce the following equation: 

\begin{definition}[Independent thinning system] \label{def:thinning}
Fix \(p\in[0,1]\). Let
\(D^p_i\) be a Poisson random measure on $\mathbb R_+\times\mathbb R_+\times\{0,1\}^{N-1}$
with intensity $\mathrm du\,\mathrm dz\nu_p(\mathrm d\theta),$, with \(\nu_p\) is the product measure of \(N-1\) independent Bernoulli
random variables with parameter \(p\). \(D^p_i\) is independent in $i$.

For \(\theta=(\theta_m)_{m\neq i}\in\{0,1\}^{N-1}\), define
\[
\mu_i(\theta)
=
B\log_2\left(
1+
\frac{1}{1+\sum_{m\neq i}\theta_m l(d(i,m))}
\right).
\]
This new system is defined the following way, for each
\(i\in\{1,\dots,N\}\), by
\begin{align}
   q_i^p(t)
=
q_i^p(0)
+
A_i[0,t]
-
\int_{\{0,1\}^{N-1}}\int_{\mathbb R_+}\int_0^t
\mathbf 1_{\{u\leq \mu_i(\theta)\}}
\mathbf 1_{\{q_i^p(z^-)>0\}}
\,D_i^p(\mathrm du,\mathrm dz,\mathrm d\theta), \label{eq:sdea}
\end{align}

\end{definition}

The process \(Q^p\) is well defined for the same reason than \eqref{eq:sde}. 
Notice that if
$\lambda < \mu^*,$
then  for any $p$  the queuing system \(Q^p\) is stable by stochastic domination with respect to the full buffer system. 

Hence \(Q^p\) has a unique stationary distribution for every $p$.

Now we want to choose $p$ such that it mimics the true system. To do so, we would like to try $p=\mathbb{P}(q^p>0)$.

\begin{theorem} \label{thm:mf2-fixedpoint}
    For $\lambda<\mu^*$, there exist a $p^* \in [0;1]$ such that \begin{align}
        p^*=\mathbb P(q_1^{p^*}>0) \label{eq:fp3}
    \end{align} where $q^p$ is the queueing system defined by \eqref{eq:sdea}.
\end{theorem}

\begin{remark}
    Here we only prove the existence of such a $p$; we did not find conditions with respect to $\lambda$ to show it is unique even if it seems to be unique numerically. 
\end{remark}

\begin{definition}[McKean--Vlasov approximation of type 2] \label{def:mf2}

The McKean--Vlasov approximation of type 2 is defined by \eqref{eq:sdea} choosing $p$ as the minimal solution of \eqref{eq:fp3} (which exists since the root set is non empty and compact). We denote it by $Q^{MF2}$.

\end{definition}

\begin{remark}\label{rem:mf2-explicit-root}
     When we simulate McKean--Vlasov approximation of type 2, we take the potential minimal solution given by the numerical solver of the fixed point: \begin{align} \label{eq:mf2-fixedpoint-explicit}
    p&=\frac{\lambda}{\sum_{k=0}^{N-1} p^k (1-p)^{N-1-k} \sum_{j=1}^{\binom{N-1}{k}}  w_j^k}.
\end{align}
    
Here, \(w_j^k\) denotes the service rate associated with the j -th configuration of $\mu_i{\theta}$, corresponding to the configuration where exactly \(k\) out of \(N-1\) Bernoulli random variables are equal to 1. 
\end{remark}
We denote $\pi^{MF2}_\lambda$ the stationary distribution of McKean--Vlasov approximation of type 2, which exists and is unique.

\begin{remark}
    One can remark that \eqref{eq:sdea} is nothing but an M/M/1 of parameter $\lambda, \mathbb{E}_{\nu_p}[\mu_1]$.
\end{remark}

Now we present the different comparison results between the different approximation systems. 

\begin{proposition}[MF2 is less congested than MF1] \label{prop:mf2-mf1}
 For $\lambda<\mu^*$,
\begin{align} \label{eq:mf2-mf1-busy}
    \pi_{\lambda}^{\mathrm{MF2}}(q > 0) \leq \pi_{\lambda}^{\mathrm{MF1}}(q > 0).
\end{align}

\end{proposition}

\begin{conjecture}[Comparison between the True system and the Mean-Field approximations] \label{Conjecture:true-bracketing}
       For $\lambda < \mu^*$, the stationary busy probability of the original system lies between those predicted by the two mean-field models:
\[
\pi_\lambda^{\mathrm{MF2}}(q_1 > 0)
\leq
\pi_\lambda(q_1 > 0)
\leq
\pi_\lambda^{\mathrm{MF1}}(q_1 > 0).
\]
In particular, Mean-Field 1 overestimates the busy probability, whereas Mean-Field 2 underestimates it.

\end{conjecture}

\subsubsection{Numerical Results}

We first illustrate the theoretical result of \Cref{thm:stability} through numerical simulations of the dynamics, using the parameter values $N=5$, $B=10$, $l=\frac{1}{r^2}$, $\mu^*=3.6$, $L=1$:

\begin{figure}[H]
    \centering
    \begin{subfigure}{0.48\textwidth}
        \centering
        \includegraphics[width=\textwidth]{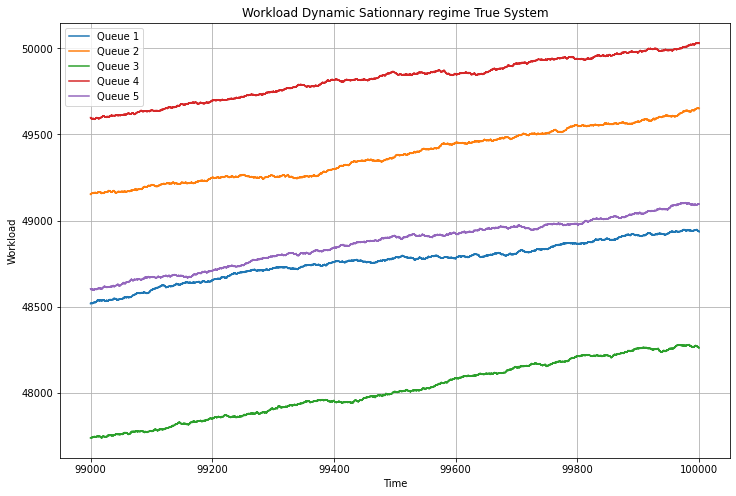}
        \caption{$\lambda=4$}
        \label{true_dynamic}
    \end{subfigure}
    \hfill
    \begin{subfigure}{0.48\textwidth}
        \centering
        \includegraphics[width=\textwidth]{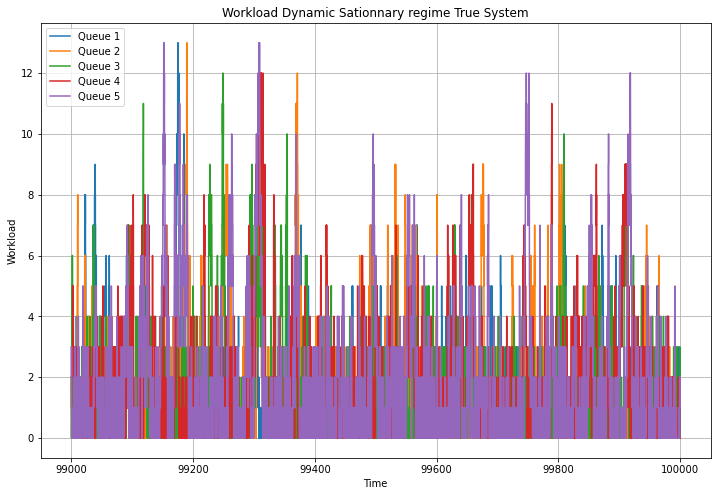}
        \caption{$\lambda=2$}
        \label{fig:wasserstein_distance}
    \end{subfigure}
    \caption{True dynamics for $\lambda$ above and below $\mu^*$}
    \label{fig:comparaison_figures}
\end{figure}

Then the following figures illustrate the solution of the fixed-point equation and report on the absolute relative difference between the mean performance metrics of the true system and those of the approximations, for both delay and congestion. They also show, for each system, the corresponding means, variances, and busy probabilities.

Here $L=1$, $B=10$, $N=5$ and Monte Carlo average are performed for time between $10^4$ and $10^5$, for the different 
path-loss models $\ell_1, \ell_2$ and $\ell_3$ with $A=1$, $C=3$, $r_0=1$.

\begin{figure}[H]
    \centering
    \includegraphics[width=0.5\textwidth]{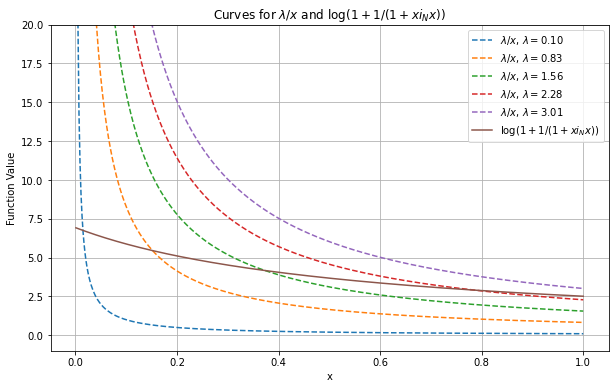}
    \caption{Solution of the fixed point equation \eqref{eq:mf1-fixedpoint} for $\ell_1$}
    \label{fig:fixed_point_1}
\end{figure}

\begin{figure}[H]
    \centering
    \includegraphics[width=0.5\textwidth]{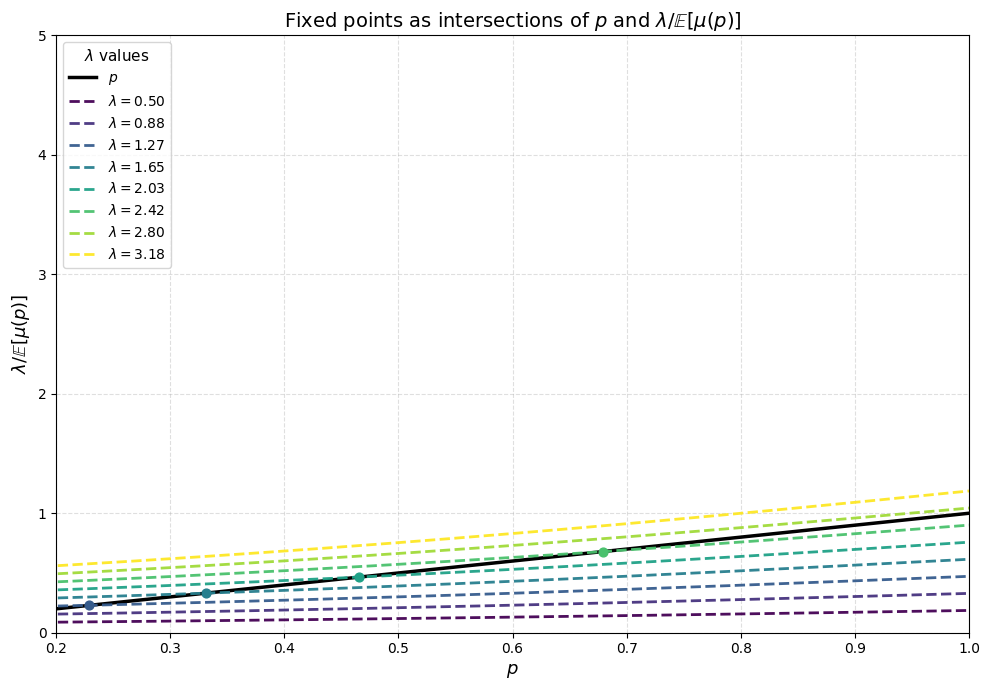}
    \caption{Solution of the fixed point equation \eqref{eq:fp3} for $\ell_1$}
    \label{fig:fixed_point}
\end{figure}

\begin{figure}[H]
    \centering
    \begin{subfigure}[b]{0.32\textwidth}
        \centering
        \includegraphics[width=\textwidth]{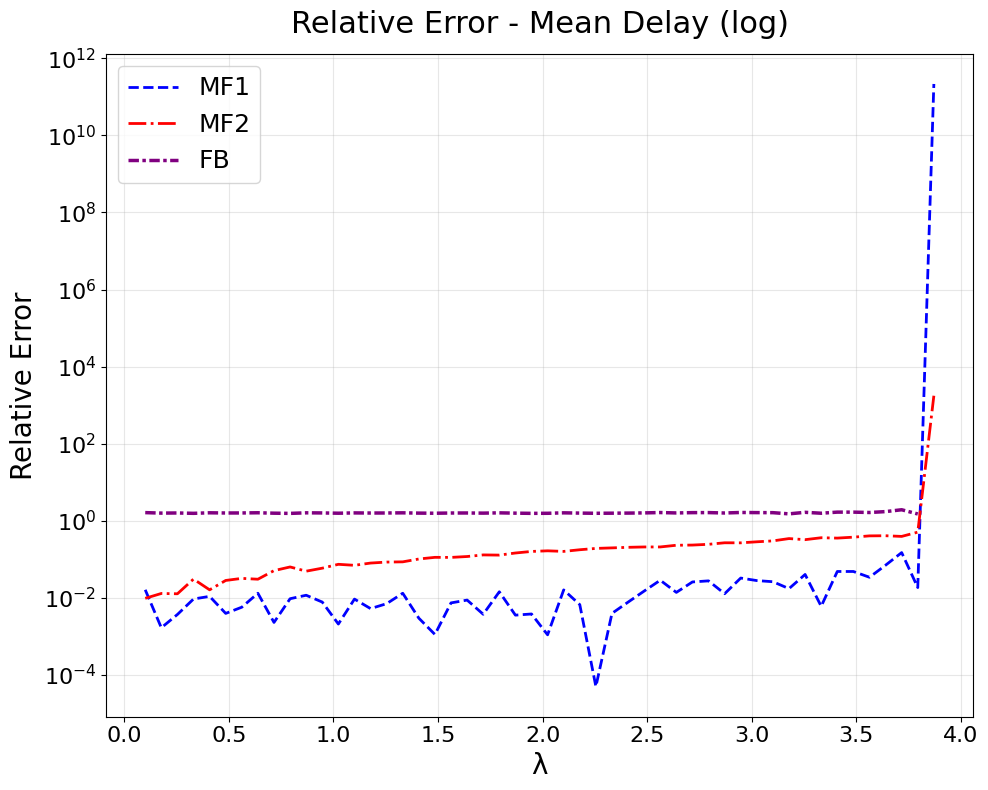}
        \caption{Relative mean difference for the delay-- $\ell_1$}
    \end{subfigure}
    \hfill
    \begin{subfigure}[b]{0.32\textwidth}
        \centering
        \includegraphics[width=\textwidth]{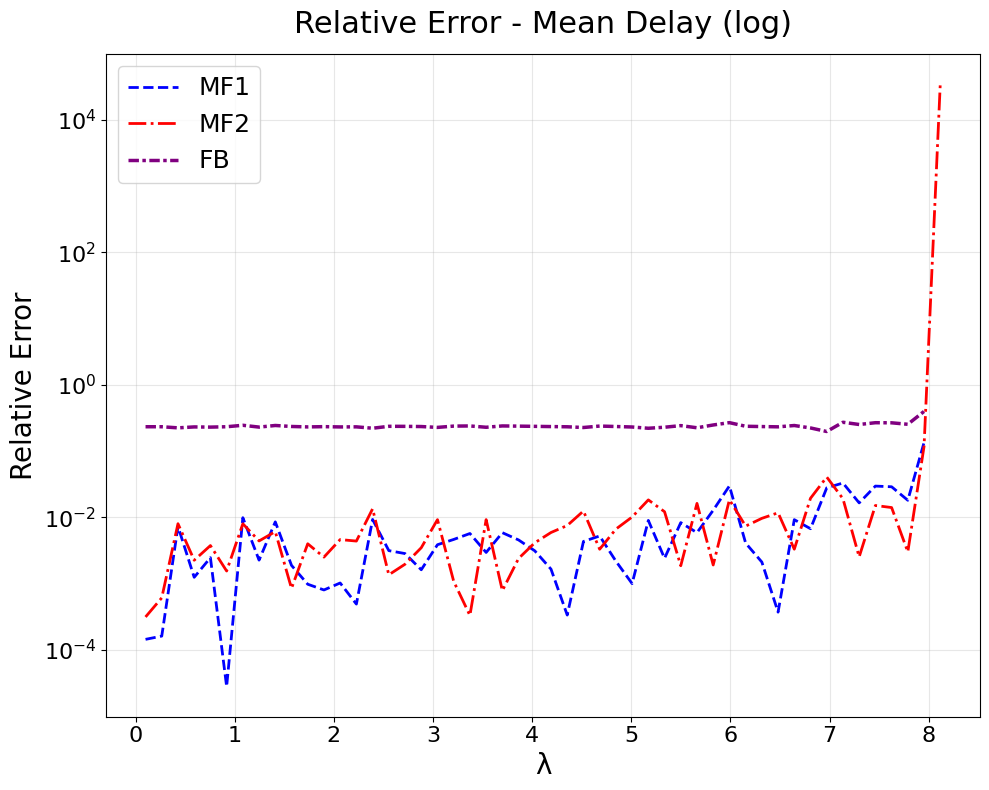}
        \caption{Relative mean difference for the delay -- $\ell_2$}
    \end{subfigure}
    \hfill
    \begin{subfigure}[b]{0.32\textwidth}
        \centering
        \includegraphics[width=\textwidth]{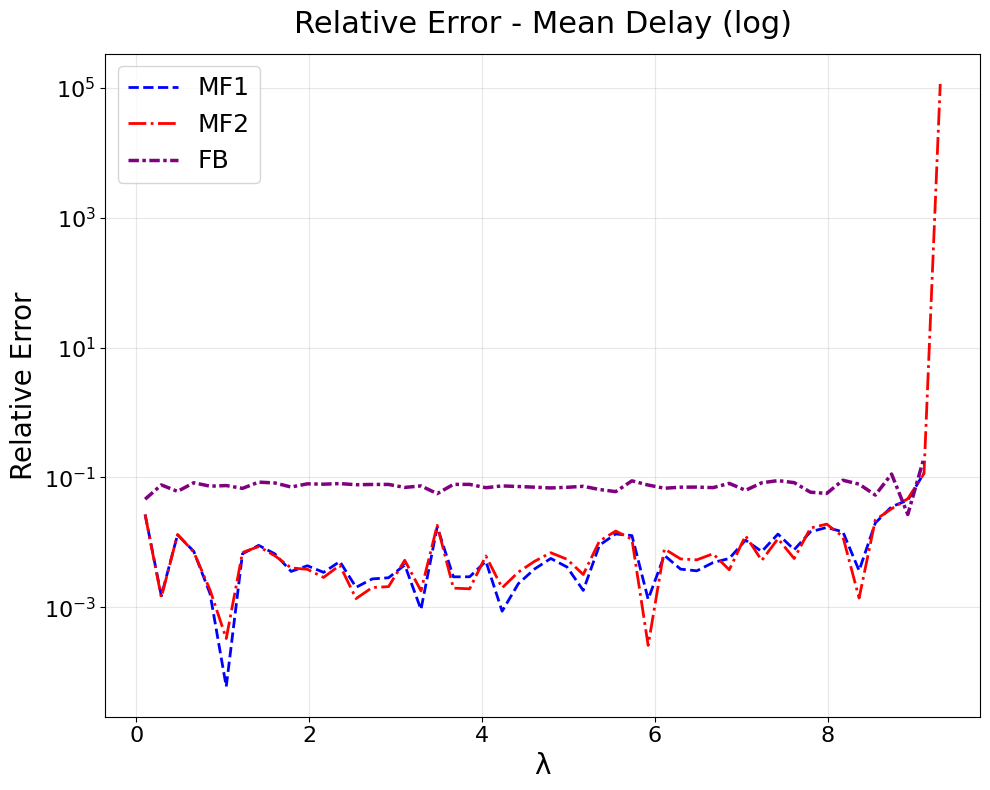}
        \caption{Relative mean difference for the delay -- $\ell_3$}
    \end{subfigure}
    \caption{Relative mean difference: Comparison of the systems,
             for path-loss models $\ell_1$, $\ell_2$, $\ell_3$.}
    \label{fig:abs_mean_distance}
\end{figure}

\begin{figure}[H]
    \centering
    \begin{subfigure}[b]{0.32\textwidth}
        \centering
        \includegraphics[width=\textwidth]{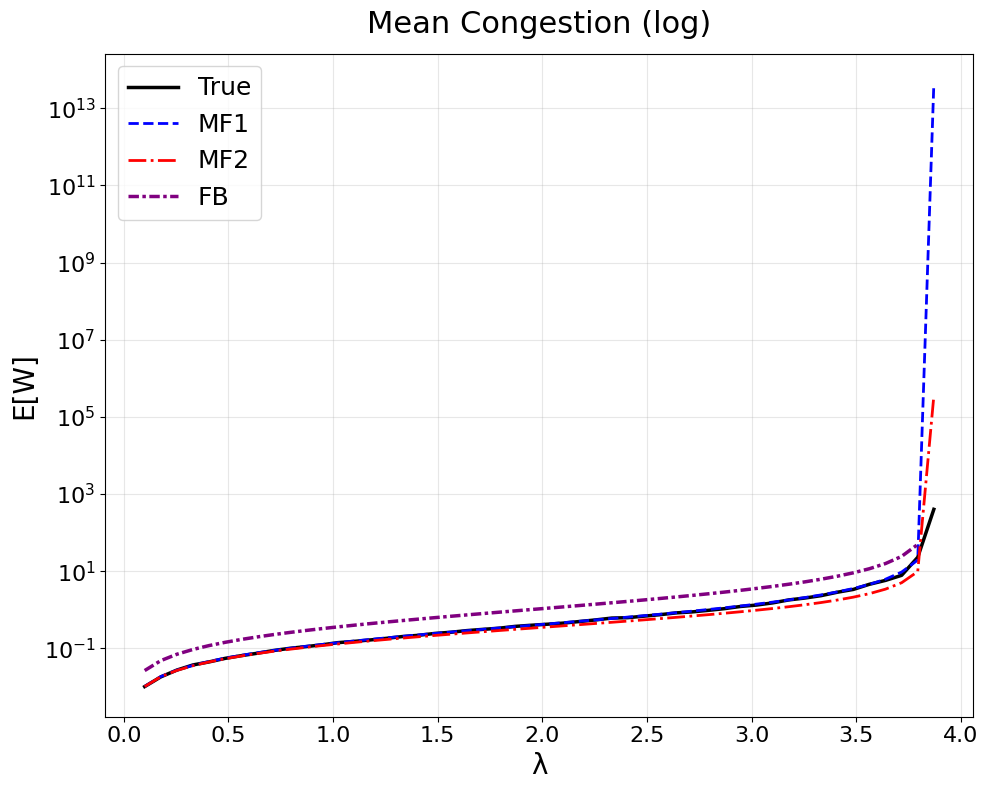}
        \caption{Mean Congestion -- $\ell_1$}
    \end{subfigure}
    \hfill
    \begin{subfigure}[b]{0.32\textwidth}
        \centering
        \includegraphics[width=\textwidth]{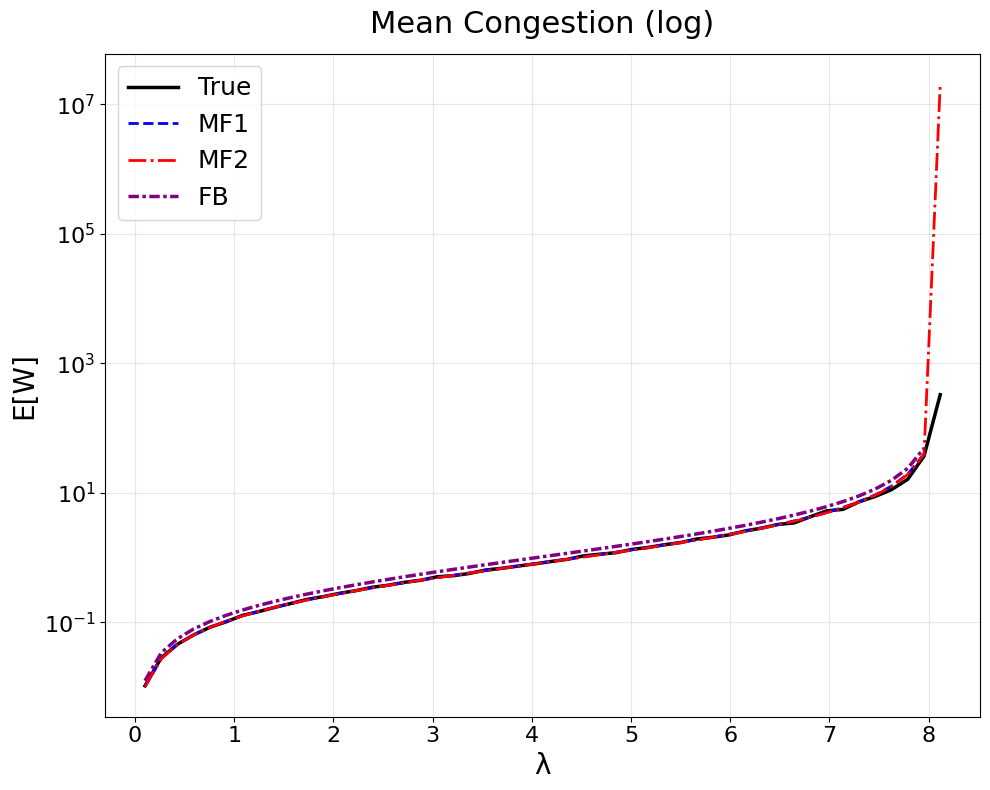}
        \caption{Mean Congestion -- $\ell_2$}
    \end{subfigure}
    \hfill
    \begin{subfigure}[b]{0.32\textwidth}
        \centering
        \includegraphics[width=\textwidth]{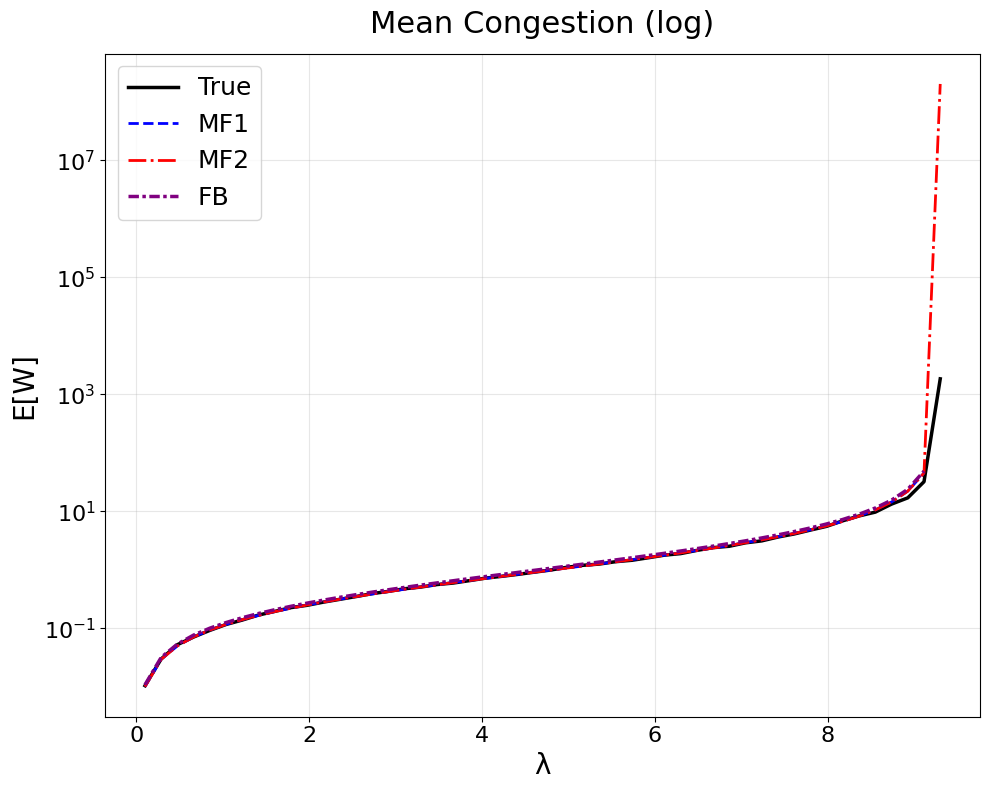}
        \caption{Mean Congestion -- $\ell_3$}
    \end{subfigure}
    \caption{Mean delay: Comparison of the systems,
             for path-loss models $\ell_1$, $\ell_2$, $\ell_3$.}
    \label{fig:mean_delay}
\end{figure}

\begin{figure}[H]
    \centering
    \begin{subfigure}[b]{0.32\textwidth}
        \centering
        \includegraphics[width=\textwidth]{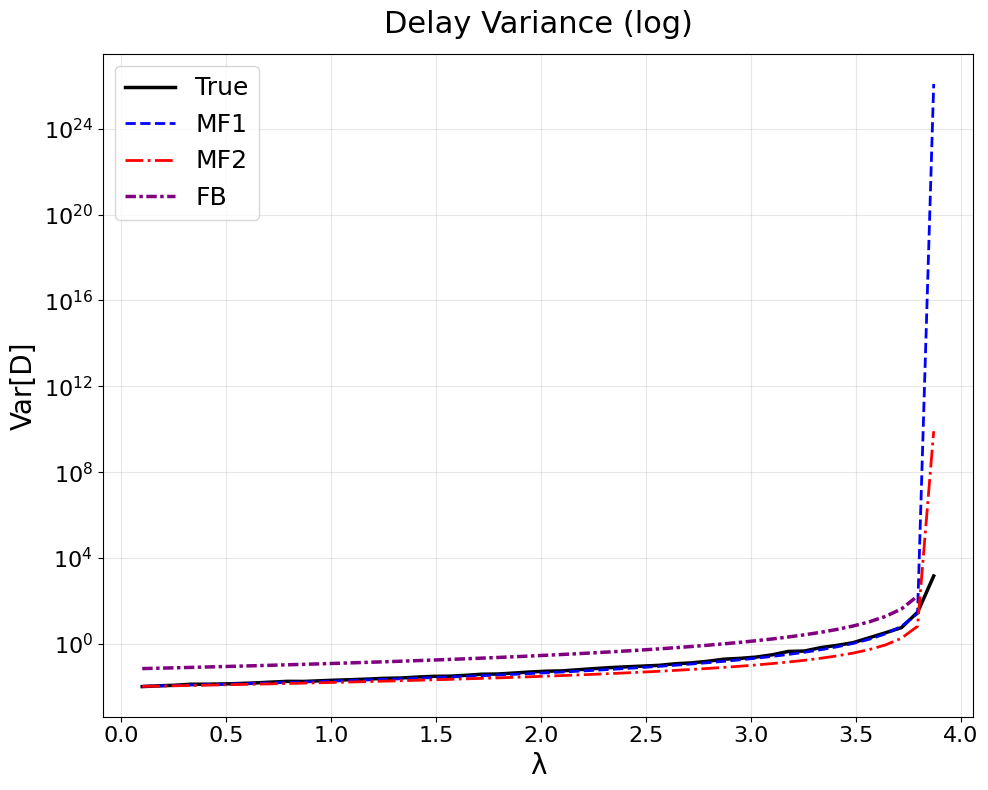}
        \caption{Variance delay -- $\ell_1$}
    \end{subfigure}
    \hfill
    \begin{subfigure}[b]{0.32\textwidth}
        \centering
        \includegraphics[width=\textwidth]{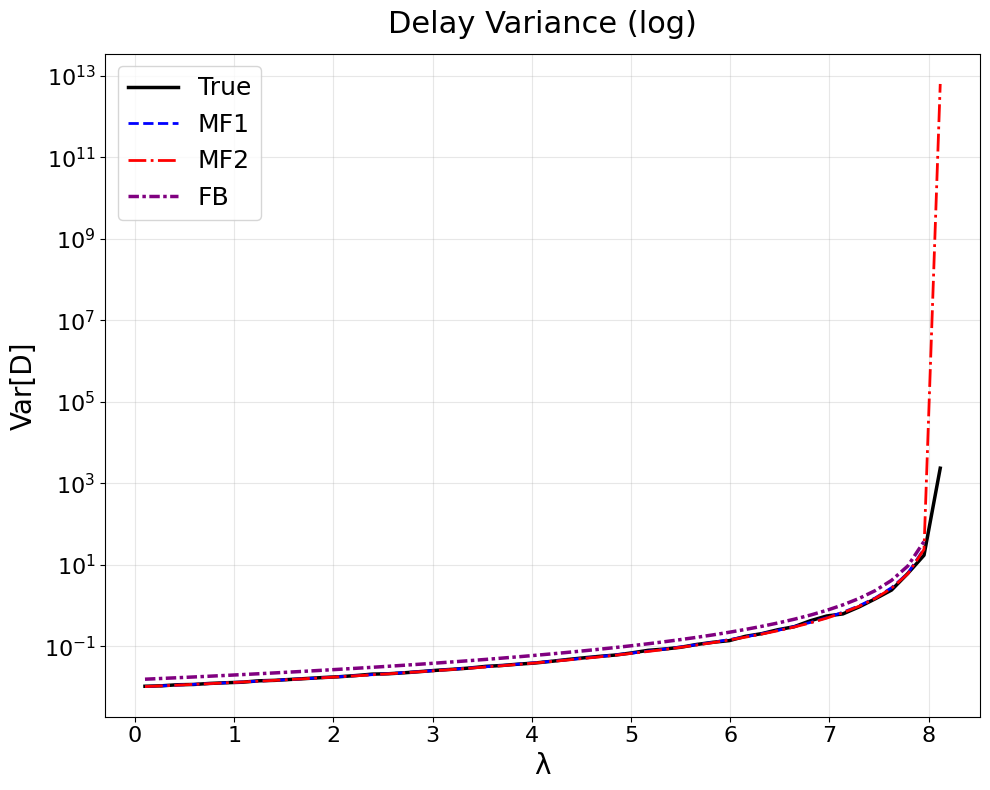}
        \caption{Variance delay -- $\ell_2$}
    \end{subfigure}
    \hfill
    \begin{subfigure}[b]{0.32\textwidth}
        \centering
        \includegraphics[width=\textwidth]{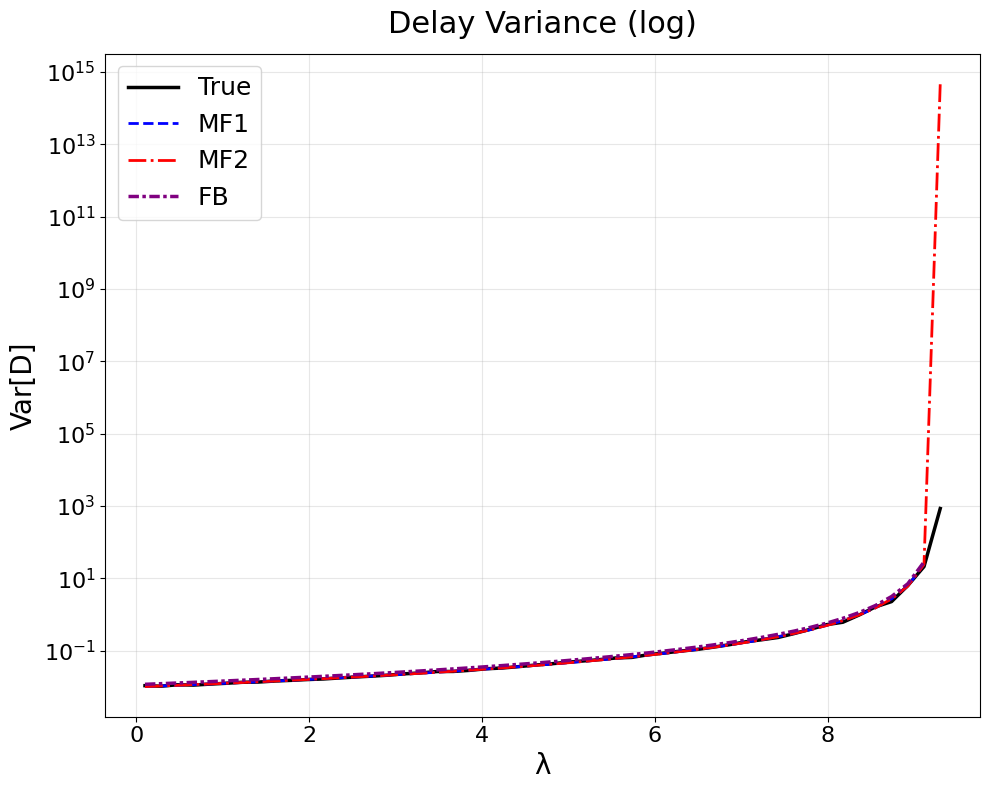}
        \caption{Variance delay -- $\ell_3$}
    \end{subfigure}
    \caption{Variance of the delay: Comparison of the systems,
             for path-loss models $\ell_1$, $\ell_2$, $\ell_3$.}
    \label{fig:var_delay}
\end{figure}

\begin{figure}[H]
    \centering
    \begin{subfigure}[b]{0.32\textwidth}
        \centering
        \includegraphics[width=\textwidth]{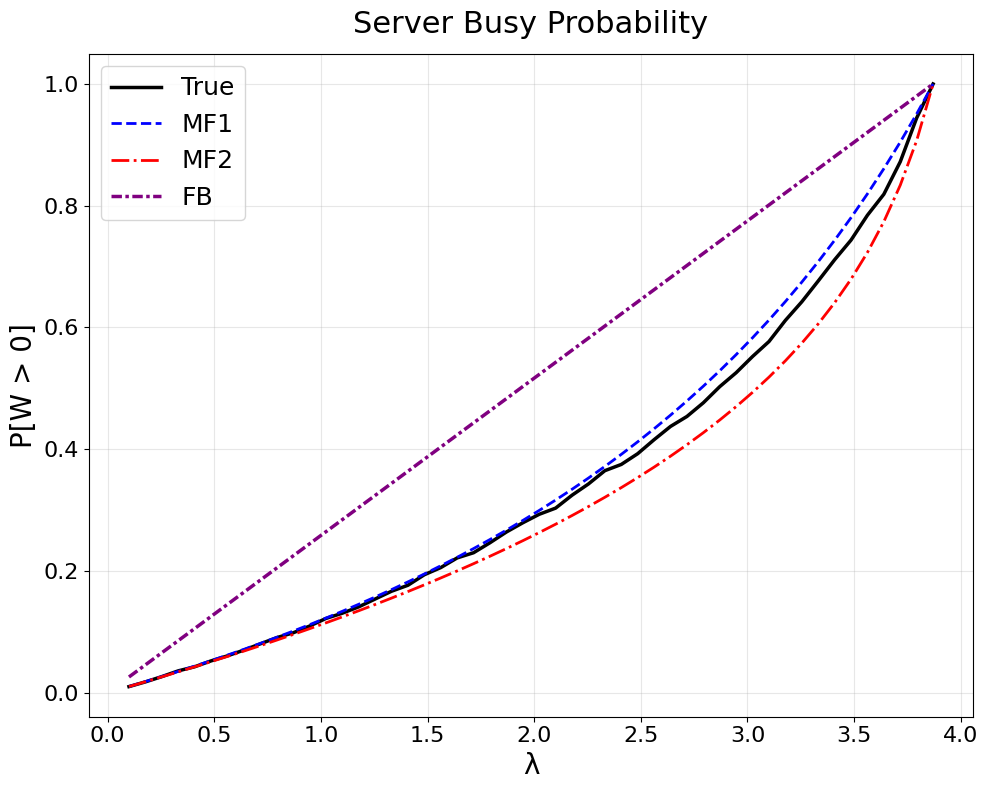}
        \caption{Busy probability -- $\ell_1$}
    \end{subfigure}
    \hfill
    \begin{subfigure}[b]{0.32\textwidth}
        \centering
        \includegraphics[width=\textwidth]{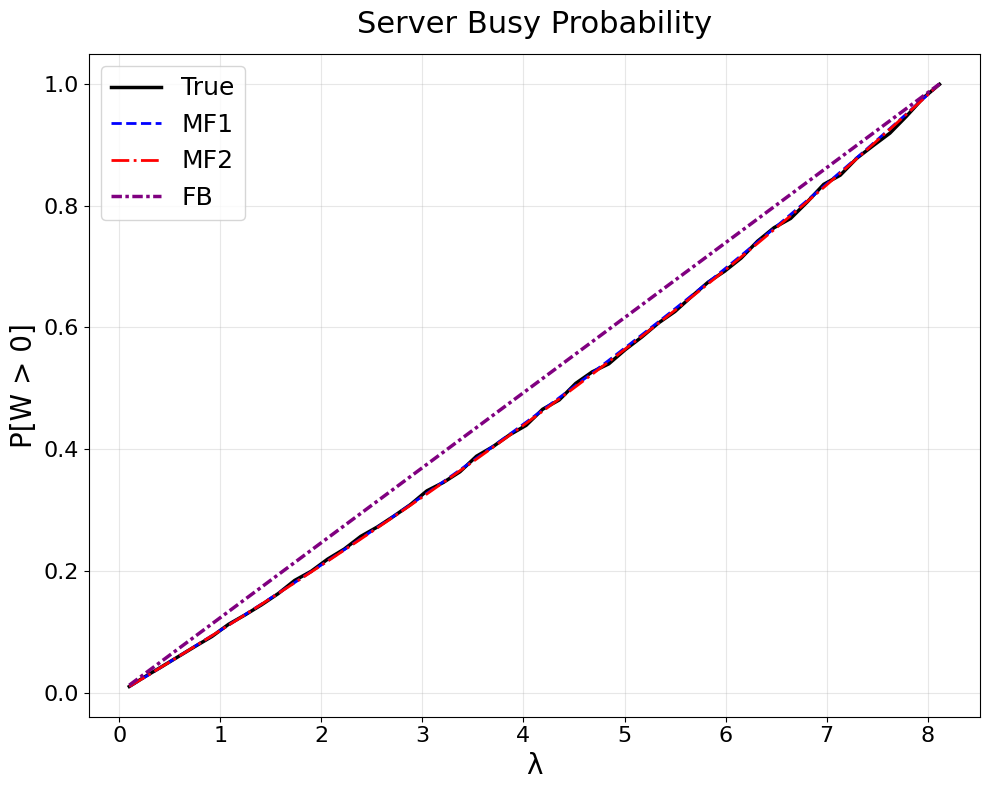}
        \caption{Busy probability -- $\ell_2$}
    \end{subfigure}
    \hfill
    \begin{subfigure}[b]{0.32\textwidth}
        \centering
        \includegraphics[width=\textwidth]{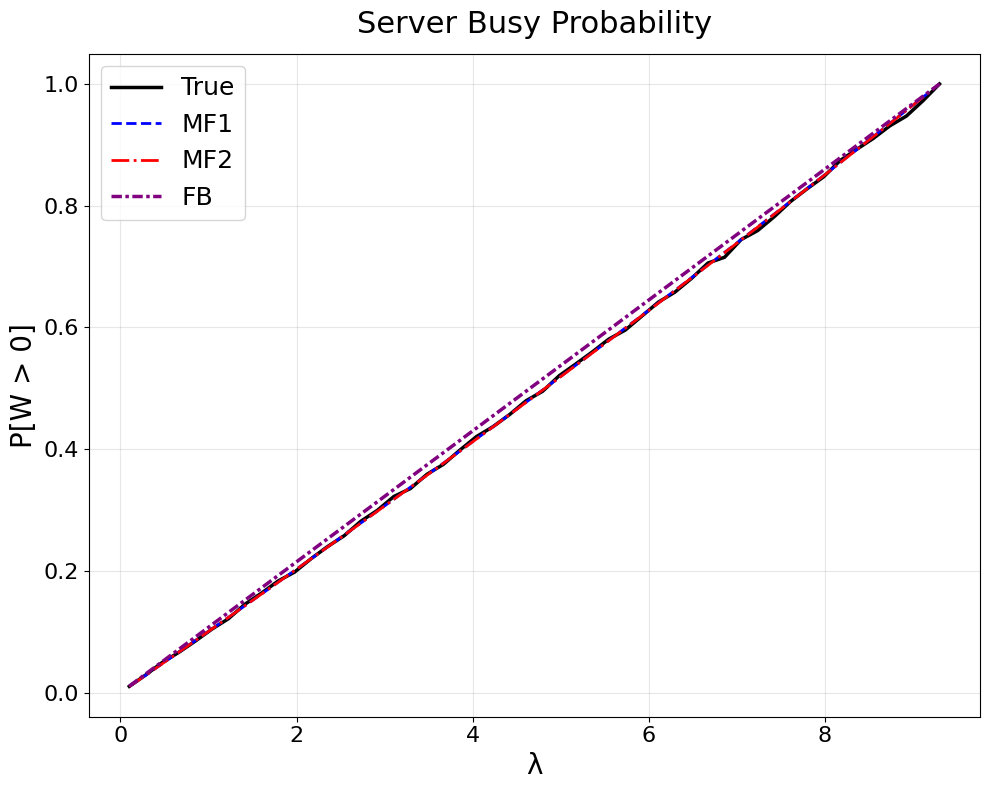}
        \caption{Busy probability-- $\ell_3$}
    \end{subfigure}
    \caption{Busy probability: Comparison of the systems,
             for path-loss models $\ell_1$, $\ell_2$, $\ell_3$.}
    \label{fig:mean_workload}
\end{figure}

The figures are obtained using the $M/M/1(\lambda,\mu)$ formulas: $
\mathbb{E}[Q] = \frac{\lambda}{\mu - \lambda},
\operatorname{Var}(D) = \frac{1}{(\mu - \lambda)^2}$, $
\mathbb{P}(Q>0) = \rho = \frac{\lambda}{\mu}.$
 For those for which there is no closed form, i.e the true system, we used Monte-Carlo average to get the different statistics. 
Figures 3 and 4 provide numerical evidence for the existence and uniqueness of the solutions to the fixed-point equations associated with the two mean-field systems, in the regime $\lambda < \mu^*$. These observations are consistent with the theoretical results. They also support our conjecture that uniqueness holds for equation equation~\eqref{eq:fp3}, although we were not able to prove it analytically.

We then compare the different approximations in terms of the mean and variance of both delay and congestion. A clear discrepancy can be observed between the full-buffer system and the original system.
In particular, for the path-loss function $\ell_1$, shown in Figure~6(a), the mean delays differ by approximately half an order of magnitude. Among the proposed approximations, the first mean-field model performs best, especially away from the stability boundary. In particular, its relative error for the mean delay remains around $10^{-2}$ for all values of $\lambda$ sufficiently far from the stability threshold.

The differences between the approximation schemes become less pronounced as the interference level decreases. Indeed, the path-loss function $\ell_1$ induces stronger interference than $\ell_2$, while $\ell_3$ corresponds to a regime where interference is comparatively weak. Consequently, in scenarios where the main quantities of interest are the mean and variance of delay and congestion, and where interference is low, all approximation schemes appear to provide satisfactory estimates.

 For the busy probability, the full-buffer system remains significantly different from the true system, with a discrepancy of roughly a factor two for path-loss functions $\ell_1$. We see on Figure 8.a that the Conjecture 3.27 seems to hold:  Mean-field 2 represent a lower bound while Mean-field 1 is an upper bound.

\subsection{Long Time behavior}\label{sec:long_time}

Since the process evolves on a countable state space and is irreducible, and
since the existence and uniqueness of an invariant distribution have already
been established, convergence to equilibrium follows from standard Markov chain
theory. The aim of this section is to obtain a quantitative estimate on this
convergence.

\begin{theorem}[Exponential convergence of the true system] \label{thm:exponential-ergodicity}
Assume that \(\lambda<\mu^*\). Let \(\pi_\lambda\) denote the unique invariant
distribution of the true system. Then, for every initial state
\(x\in\mathbb N^N\), there exist constants \(C_x<\infty\) and \(\kappa>0\) such
that
\[
\|\delta_x P_t-\pi_\lambda\|_{TV}
\leq C_x e^{-\kappa t},
\qquad t\geq 0.
\]
\end{theorem}

\subsection{Random positions}\label{sec:rd_pos}

We were dealing with non random position. Now let's take this into account and try to get estimate on principal parameter thanks to the Replica Mean-field machinery. Since we are in the equi-distant setting the randomness will be on the spacing of the queues. We will take it uniformly distributed on $[a,b]$ with $b>a>0$ and use the power law path loss function $l(x)=x^{-\alpha}$. Here we use transmit power $P_e$ and not the normalized version. We choose $N$ being odd to simplify notation.
Take: 
\begin{itemize}
    \item $\psi=\sum_{k=0}^{N-1}\delta_{Sk}.$
    \item $S \sim $ Unif $(a,b)$.
    \item $J_N:=2P_e\sum\limits_{j \in [1;(N-1)/2]}j^{-\alpha}$
\end{itemize}
This implies a common distance $S$ between the points. 
First let's compute the average worst speed $\mathbb{E}[\mu^*]$.
\begin{align} \label{eq:random-spacing-exact-mean}
    \mathbb{E}[\mu^*]&=\mathbb{E}[B\log_2(1+\frac{P_e}{1+J_NS^{-\alpha}})] \\
    &=\frac{B}{\ln(2)(b-a)}\int_{a}^b\log(1+\frac{P_e}{1+J_Ns^{-\alpha}})\rm{d}s \nonumber\\
    &=-\frac{B}{\ln(2)\alpha(b-a)}J_N^{1/\alpha}\int_{\frac{J_N}{a^\alpha}}^{\frac{J_N}{b^\alpha}}\frac{1}{u^{\frac{\alpha+1}{\alpha}}}\log(1+\frac{P_e}{1+u})\rm{d}u \nonumber\\
    &=-\frac{B}{\ln(2)\alpha(b-a)}J_N^{1/\alpha}\Bigg[[\frac{-\alpha}{u^{1/\alpha}}\log(1+\frac{P_e}{1+u})]_\frac{J_N}{a^\alpha}^\frac{J_N}{b^\alpha} \notag \\
    &-\int_{\frac{J_N}{a^\alpha}}^{\frac{J_N}{b^\alpha}}\frac{\alpha}{u^{1/\alpha}}(\frac{1}{1+u}-\frac{1}{1+u+P_e})\rm{d}u\Bigg].
\end{align}
Depending on the choice of positive rational $\alpha$, one can continue the computation thanks to a partial fraction decomposition and change of variable and have final expression in terms of the logarithm and arctan functions. 

Now we want $\mathbb{E}[q]$ using the Mean-field approximation: $\mathbb{E}[q] \approx \mathbb{E}^{MF1}[q]$
and: 
\begin{align}
&\mathbb{E}^{\mathrm{MF1}}[q]
=
\mathbb{E}^{\mathrm{MF1}}
\!\left[
\mathbb{E}^{\mathrm{MF1}}[q\mid S]
\right]
\notag \\
&\mathbb{E}^{\mathrm{MF1}}
\!\left[
\mathbb{E}^{\mathrm{MF1}}[q\mid S]
\right]
=
\mathbb{E}^{\mathrm{MF1}}
\!\left[
\frac{x^{\star}(S)}{1 - x^{\star}(S)}
\right].
\end{align}
Unfortunately $x^*$ has no closed form but we can approximate it when the SINR is low. To do so we make the transmit power go to 0 and place our self a fixed relative distance below the stability boundary.

   \begin{proposition}[Random spacing approximation]
\label{prop:random-spacing}
For every $\epsilon>0$ and $\delta\in(0,1)$, there exists
$P_0>0$ such that, for every $P_e\in(0,P_0)$ and every
$0<\lambda<(1-\delta)B\log_2(1+\frac{P_e}{1+\frac{J_N}{a^\alpha}})$,
\[
\left|
\mathbb E^{\mathrm{MF1}}[q]-A(P_e,\lambda)
\right|\leq\epsilon.
\]

whith \[A(P_e,\lambda) =\frac{\lambda}{(b-a)\alpha}J_N^{1/\alpha}\int_{\frac{J_N}{b^\alpha}}^{\frac{J_N}{a^\alpha}}\frac{1}{BP_e/\ln(2)-\lambda(1+ u)}\frac{1}{u^{1+1/\alpha}}\rm{d}u,\]
\end{proposition}

We can get closed form expressions for certain values of $\alpha$.
With Little's law, we get for free the expected delay of typical job: $\mathbb{E}^{MF1,0}[D]=\frac{1}{\lambda}\mathbb{E}^{MF}[q] $.


\section{Proofs}\label{sec:proof}

In this section we dive into the details of the proofs.  To ease the notations we sometimes replace $B\log_2(1+\frac{1}{1+x})$ by $g(x)$. 

\subsection{Stability Region}
First we can bound the stability region using $\mu^*$ and $\mu^{**}$ defined in \Cref{sec:descr}, 
\begin{lemma} \label{lem:stability-bounds}
    The stability region of Q, i.e the range of $\lambda$ for which the chain Q is positive recurrent, is at least $[0,\mu^*)$ and at most $[0,\mu^{**})$.
\end{lemma}

\begin{proof}
    Use the monotonicity property of the system proved in \Cref{lem:mono} for the departure rate being constant equal to $\mu^*$ and $\mu^{**}$
\end{proof}

We are now able to prove the exact stability region.
\begin{proof}[Proof of \Cref{thm:stability}]
The case $\lambda<\mu^*$ follows from \Cref{lem:stability-bounds}. We now consider the case
$\lambda>\mu^*$.

Let
\[
    \partial := \{q\in\mathbb N^N:\exists i,\ q_i=0\},
    \qquad
    \tau_\partial := \inf\{t\geq 0: Q(t)\in\partial\}.
\]
We shall prove that there exists $n\geq 1$ such that
\[
    \mathbb P_{(n,\dots,n)}(\tau_\partial=\infty)>0.
\]
This is incompatible with recurrence. Indeed, for an irreducible Markov
process on a countable state space, recurrence implies that every state is
hit almost surely from every other state. In particular, starting from
$(n,\dots,n)$, the process would hit a state in $\partial$ almost surely.

We construct on the same probability space as $Q$ a process
$X=(X_1,\dots,X_N)$, where the coordinates are independent $M/M/1$ queues
with arrival rate $\lambda$ and service rate $\mu^*$. Namely, for
$i=1,\dots,N$,
\[
    X_i(t)
    =
    X_i(0)+A_i(t)
    -
    \int_0^t\int_0^\infty
    \mathbf 1_{\{X_i(s-)>0\}}
    \mathbf 1_{\{u\leq \mu^*\}}
    D_i(ds,du),
\]
with
\[
    X(0)=Q(0)=(n,\dots,n).
\]
The coordinates of $X$ are independent since they are driven by independent
arrival processes and independent potential service processes.

On the time interval $[0,\tau_\partial)$, all coordinates of $Q$ are strictly
positive. Hence, for every $i$ and every $s<\tau_\partial$,
\[
    \mu_i(Q(s-))=\mu^*.
\]
Therefore, by pathwise uniqueness for the jump equation driven by the same
Poisson processes,
\[
    Q(t)=X(t),\qquad 0\leq t<\tau_\partial.
\]
Consequently,
\[
    \{\tau_\partial=\infty\}
    =
    \bigcap_{i=1}^N \{\tau_0^{(i)}=\infty\},
\]
where
\[
    \tau_0^{(i)}:=\inf\{t\geq 0:X_i(t)=0\}.
\]
Thus, using the independence of the coordinates of $X$,
\[
    \mathbb P_{(n,\dots,n)}(\tau_\partial=\infty)
    =
    \prod_{i=1}^N
    \mathbb P_n(\tau_0^{(i)}=\infty).
\]
Each $X_i$ is an $M/M/1$ queue with arrival rate $\lambda$ and service rate
$\mu^*$. Since $\lambda>\mu^*$, the associated birth--death chain has positive
drift and
\[
    \mathbb P_n(\tau_0^{(i)}=\infty)
    =
    1-\left(\frac{\mu^*}{\lambda}\right)^n
    >0.
\]
Therefore,
\[
    \mathbb P_{(n,\dots,n)}(\tau_\partial=\infty)
    =
    \left[
    1-\left(\frac{\mu^*}{\lambda}\right)^n
    \right]^N
    >0.
\]
Hence $Q$ cannot be recurrent. Since the process is irreducible on a
countable state space, it follows from the standard classification that $Q$
is transient.
\end{proof}

We now prove that the system is associated. 

\begin{proof}[Proof of \Cref{prop:association}]

Fix $x\in\mathbb N^N$ and let $M\geq \|x\|_\infty$. We define a truncated process
$Q^{(M)}$ on the finite state space
\[
E_M:=\{0,\dots,M\}^N.
\]
The process $Q^{(M)}$ has the same dynamics as $Q$, except that arrivals on
coordinate $i$ are suppressed whenever the corresponding queue is already equal
to $M$. More precisely, $Q^{(M)}(0)=x$ and, for every $i=1,\dots,N$,
\begin{align} \label{eq:finite-capacity}
q_i^{(M)}(t)
&=
q_i^{(M)}(0)
+
\int_{(0,t]}
\mathbf 1_{\{q_i^{(M)}(s^-)<M\}}\,A_i(\mathrm ds)
\nonumber \\
&\quad
-
\int_0^t\int_{\mathbb R_+}
\mathbf 1_{\{q_i^{(M)}(s^-)>0\}}
\mathbf 1_{\{u\leq \mu_i(Q^{(M)}(s^-))\}}
\,D_i(\mathrm du,\mathrm ds).
\end{align}

The process $Q^{(M)}$ is a continuous-time Markov chain on the finite partially
ordered space $E_M$. Moreover, by the same coupling argument as in
Lemma~\ref{lem:mono}, it is monotone with respect to the usual partial
order on $E_M$. Its jumps are only of the form
\[
q\longmapsto q+e_i
\qquad\text{or}\qquad
q\longmapsto q-e_i,
\]
whenever the resulting state belongs to $E_M$. Hence all jumps are between
comparable states.

We may therefore apply the preservation theorem for association for monotone
Markov chains on finite partially ordered spaces whose jumps are between
comparable states; see \cite{10.1214/aop/1176995804}. Since the deterministic
measure $\delta_x$ is associated, it follows that, for every $t\geq 0$, the law
of $Q^{(M)}(t)$ is associated. Thus, for all bounded increasing functions
$f,g:E_M\to\mathbb R$,
\[
\mathbb E_x\bigl[f(Q^{(M)}(t))g(Q^{(M)}(t))\bigr]
\geq
\mathbb E_x\bigl[f(Q^{(M)}(t))\bigr]\,
\mathbb E_x\bigl[g(Q^{(M)}(t))\bigr].
\]

We now let $M\to\infty$. Construct $Q^{(M)}$ and $Q$ on the same probability
space, using the same arrival and departure Poisson measures. Since $Q$ is
non-explosive, for every fixed $t\geq 0$,
\[
\sup_{0\leq s\leq t}\max_{1\leq i\leq N} q_i(s)<\infty
\qquad\text{a.s.}
\]
Consequently, for almost every sample path and for all $M$ sufficiently large,
the upper boundary $M$ is never reached by $Q$ on $[0,t]$. Therefore the
truncated and non-truncated dynamics coincide on $[0,t]$, and in particular
\[
Q^{(M)}(t)\longrightarrow Q(t)
\qquad\text{a.s.}
\]

Let $f,g:\mathbb N^N\to\mathbb R$ be bounded increasing functions. Applying the
previous association inequality to the restrictions of $f$ and $g$ to $E_M$,
then passing to the limit by dominated convergence theorem, gives
\[
\mathbb E_x\bigl[f(Q(t))g(Q(t))\bigr]
\geq
\mathbb E_x\bigl[f(Q(t))\bigr]\,
\mathbb E_x\bigl[g(Q(t))\bigr].
\]
Thus the law of $Q(t)$ under $\mathbb P_x$ is associated for every $t\geq 0$.
\end{proof}

\begin{proof}[Proof of \Cref{cor:largest-queue-tail}]
We have, 
\[
    \{M(t)<k\}
    =
    \bigcap_{i=1}^N \{q_i(t)<k\}.
\]
The functions
\[
    x\longmapsto \mathbf{1}_{\{x_i<k\}},
    \qquad i\in[1:N],
\]
are bounded and coordinatewise non-increasing. Association therefore
yields
\[
\begin{aligned}
    \mathbb{P}\bigl(M(t)<k\bigr)
    &=
    \mathbb{P}\left(
        \bigcap_{i=1}^N\{q_i(t)<k\}
    \right)\geq
    \prod_{i=1}^N
    \mathbb{P}\bigl(q_i(t)<k\bigr).
\end{aligned}
\]
Because the queue are identically distributed,
$
    \mathbb{P}\bigl(q_i(t)<k\bigr)
    =
    \mathbb{P}\bigl(q_1(t)<k\bigr)
$
for every \(i\). Taking complements gives
\[
    \mathbb{P}\bigl(M(t)\geq k\bigr)
    \leq
    1-
    \left(
        1-\mathbb{P}\bigl(q_1(t)\geq k\bigr)
    \right)^N.
\]

For the second statement, suppressing departures can only increase the
queue length. Hence, under the stated assumption,
$
    q_1(t)
    \leq_{\mathrm{st}}
    q_1(0)+A(t),
$
where \(A(t)\) is a Poisson random variable with parameter \(\lambda t\).
Since \(q_1(0)\) is assume to be deterministic and \(A(t)\) has finite
exponential moments of every order, their sum has an exponential tail.
Consequently, there exist \(C_t<\infty\) and \(\kappa_t>0\) such that
\[
    \mathbb{P}\bigl(q_1(t)\geq k\bigr)
    \leq
    \min\{1,C_t e^{-\kappa_t k}\}.
\]
Substituting this estimate into
\eqref{eq:largest-queue-association} proves
\eqref{eq:largest-queue-exponential-tail}.
\end{proof}

\subsection{Stationary Distribution}

 We now establish the well-posedness of \eqref{eq:MF1} and the propagation of chaos of \eqref{eq:replica-generator} toward it. We base our analysis on \cite{chevallier_champ_moyen} and \cite{fournier2015toymodelinteractingneurons}. 

\begin{lemma}\label{lem:deterministic-rate}
   Let \(T>0\) and \(m\in C([0,T],[0,1])\). Define
\[
\mu^m(t)
=
g\left(m(t)\xi_N\right),
\qquad t\in[0,T].
\]
Then the equation
\begin{align}\label{eq:qm}
q^m(t)
&=
q^m(0)
+
A[0,t]
-
\int_0^t\int_{\mathbb R_+}
\mathbf 1_{\{q^m(z^-)>0\}}
\mathbf 1_{\{u\leq \mu^m(z^-)\}}
\,D(\mathrm du,\mathrm dz)
\end{align}
admits a unique strong solution.
\end{lemma}
\begin{proof}
The two equations above define a birth--death process with an inhomogeneous, deterministic, and bounded death rate. 
For the same reasons as in \Cref{lem:sde-construction}, this process admits an SDE representation.
\end{proof}

We now define the map
\[
\begin{aligned}
\Phi : C([0,T],[0,1]) &\longrightarrow C([0,T],[0,1]) \\
m &\longmapsto \Phi(m),
\end{aligned}
\]
where
\[
\Phi(m)(t)
:=
\mathbb P\left(q^m(t)>0\right),
\qquad t\in[0,T].
\]
We first check that \(\Phi\) is well defined.
\begin{lemma}\label{lem:phi-continuous}
For every \(m\in C([0,T],[0,1])\), the map
\[
t\longmapsto \mathbb P\left(q^m(t)>0\right)
\]
is Lipschitz continuous. In particular, \(\Phi(m)\in C([0,T],[0,1])\).
\end{lemma}

\begin{proof}
The generator of \(q^m\) at time \(t\) is given, for bounded functions \(h\), by
\[
\mathcal L_t^m h(x)
=
\lambda\big(h(x+1)-h(x)\big)
+
\mu^m(t)\mathbf 1_{\{x>0\}}
\big(h(x-1)-h(x)\big).
\]
Let
\[
h(x)=\mathbf 1_{\{x>0\}}.
\]
By Dynkin's formula, for \(0\leq s\leq t\leq T\),
\[
\mathbb E[h(q^m(t))]
-
\mathbb E[h(q^m(s))]
=
\int_s^t
\mathbb E\left[
\mathcal L_u^m h(q^m(u))
\right]
\,\mathrm du.
\]
Moreover,
\[
\left|\mathcal L_u^m h(x)\right|
\leq
\lambda + \|\mu^m\|_{\infty,[0,T]}.
\]
Therefore,
\[
\left|
\mathbb P(q^m(t)>0)
-
\mathbb P(q^m(s)>0)
\right|
\leq
\left(\lambda+\|\mu^m\|_{\infty,[0,T]}\right)|t-s|.
\]
Thus \(t\mapsto \mathbb P(q^m(t)>0)\) is Lipschitz continuous.
\end{proof}

We next prove that \(\Phi\) is a contraction on a sufficiently small time
interval.

\begin{lemma}\label{lem:phi-contraction}
There exists \(\widetilde T>0\) such that \(\Phi\) is a contraction on
\(C([0,\widetilde T],[0,1])\).
\end{lemma}

\begin{proof}
Let \(m_1,m_2\in C([0,T],[0,1])\). We construct \(q^{m_1}\) and \(q^{m_2}\) on
the same probability space, using the same arrival process \(A\), the same
departure Poisson point measure \(D\), and the same initial condition.

For every \(t\in[0,T]\),
\begin{align*}
\left|\Phi(m_1)(t)-\Phi(m_2)(t)\right|
&=
\left|
\mathbb P(q^{m_1}(t)>0)
-
\mathbb P(q^{m_2}(t)>0)
\right| \\
&\leq
\mathbb E\left[
\left|
\mathbf 1_{\{q^{m_1}(t)>0\}}
-
\mathbf 1_{\{q^{m_2}(t)>0\}}
\right|
\right].
\end{align*}
Since \(q^{m_1}(t),q^{m_2}(t)\in\mathbb N\), we have
\[
\left|
\mathbf 1_{\{q^{m_1}(t)>0\}}
-
\mathbf 1_{\{q^{m_2}(t)>0\}}
\right|
\leq
\left|q^{m_1}(t)-q^{m_2}(t)\right|.
\]
Hence
\[
\left|\Phi(m_1)(t)-\Phi(m_2)(t)\right|
\leq
\mathbb E\left[
\left|q^{m_1}(t)-q^{m_2}(t)\right|
\right].
\]

Define the symmetric difference of the two thinning regions up to time \(T\) by
\[
\Delta_{m_1,m_2}(T)
:=
\left\{
(u,z)\in \mathbb R_+\times[0,T]:
\mu^{m_1}(z)\wedge \mu^{m_2}(z)
<
u
\leq
\mu^{m_1}(z)\vee \mu^{m_2}(z)
\right\}.
\]
Under the above coupling, the two processes may differ only because of departure
marks falling in \(\Delta_{m_1,m_2}(T)\). Thus, pathwise,
\[
\left|q^{m_1}(t)-q^{m_2}(t)\right|
\leq
D\left(\Delta_{m_1,m_2}(T)\right),
\qquad t\in[0,T].
\]
Indeed, common arrivals affect both processes in the same way, and departure
marks lying in the common thinning region are used simultaneously, except
possibly when one of the queues is empty; in that case such a departure mark
cannot increase the discrepancy. Therefore only marks in the symmetric
difference of the thinning regions can increase the discrepancy.

Consequently,
\begin{align*}
\left|\Phi(m_1)(t)-\Phi(m_2)(t)\right|
&\leq
\mathbb E\left[D\left(\Delta_{m_1,m_2}(T)\right)\right] \\
&=
\left|\Delta_{m_1,m_2}(T)\right| \\
&=
\int_0^T
\left|\mu^{m_1}(z)-\mu^{m_2}(z)\right|
\,\mathrm dz .
\end{align*}
Since \(g\) is \(L\)-Lipschitz,
\[
\left|\mu^{m_1}(z)-\mu^{m_2}(z)\right|
=
\left|
g(m_1(z)\xi_N)-g(m_2(z)\xi_N)
\right|
\leq
L\xi_N |m_1(z)-m_2(z)|.
\]
Therefore,
\[
\left|\Phi(m_1)(t)-\Phi(m_2)(t)\right|
\leq
T L\xi_N
\|m_1-m_2\|_{\infty,[0,T]}.
\]
Taking the supremum over \(t\in[0,T]\), we obtain
\[
\|\Phi(m_1)-\Phi(m_2)\|_{\infty,[0,T]}
\leq
T L\xi_N
\|m_1-m_2\|_{\infty,[0,T]}.
\]
Hence \(\Phi\) is a contraction on \(C([0,\widetilde T],[0,1])\) as soon as
\[
\widetilde T L\xi_N < 1.
\]
This proves the claim.
\end{proof}

\begin{proposition}\label{prop:local-well-posedness}
The McKean--Vlasov equation \eqref{eq:MF1} admits a unique strong solution on
\([0,\widetilde T]\), for every \(\widetilde T>0\) such that
$
\widetilde T L\xi_N<1.
$
\end{proposition}

\begin{proof}
By Lemma~\ref{lem:phi-contraction}, \(\Phi\) is a contraction on
\(C([0,\widetilde T],[0,1])\). Since this space is complete for the uniform norm,
Banach's fixed point theorem yields a unique \(m^*\in C([0,\widetilde T],[0,1])\)
such that
\[
\Phi(m^*)=m^*.
\]
Equivalently,
\[
m^*(t)=\mathbb P(q^{m^*}(t)>0),
\qquad t\in[0,\widetilde T].
\]
The process \(q^{m^*}\) therefore satisfies
\[
q^{m^*}(t)
=
q^{m^*}(0)
+
A[0,t]
-
\int_0^t\int_{\mathbb R_+}
\mathbf 1_{\{q^{m^*}(z^-)>0\}}
\mathbf 1_{\{u\leq g(m^*(z^-)\xi_N)\}}
\,D(\mathrm du,\mathrm dz),
\]
with
\[
m^*(t)=\mathbb P(q^{m^*}(t)>0).
\]
Thus \(q^{m^*}\) is a strong solution of the McKean--Vlasov equation on
\([0,\widetilde T]\). Uniqueness follows from the uniqueness of the fixed point.
\end{proof}

\begin{proof}[Proof of \Cref{thm:mf1-wellposed}]
The previous argument is local in time. However, the contraction time
\(\widetilde T\) depends only on \(L\) and \(\xi_N\), and not on the initial
condition. Therefore, after constructing the solution on \([0,\widetilde T]\),
one can restart the argument at time \(\widetilde T\), using
\(q^{m^*}(\widetilde T)\) as initial condition.

Repeating this procedure on successive intervals of length \(\widetilde T\)
yields existence and uniqueness on any finite time interval \([0,T]\). Hence the
McKean--Vlasov equation \eqref{eq:MF1} admits a strong solution and is pathwise
unique.
\end{proof}

\begin{proof}[Proof of \Cref{thm:chaos-one}]
   We define $
I_i^{(MF1,r)}(t)
:=
\sum_{j\neq i}
\mathbf 1_{\{q_j^{(MF1,r)}(t)>0\}}
l(d(i,j))
$ and remark that $
\mu^{MF1}(t)
=
g\left(
\mathbb E\left[I_i^{(MF1,1)}(t)\right]
\right).
$

Fix \(T>0\) and \(i\in\{1,\dots,N\}\). Since the two processes
\(q_i^{(K,1)}\) and \(q_i^{(MF1,1)}\) are driven by the same arrivals and the
same Poisson point measure, the arrivals and the initial conditions cancel in the difference.

Moreover, departures accepted by both processes simultaneously cannot increase
the distance between the two queue lengths. Hence the two processes can differ
only through Poisson points which are accepted by one of the two thinning
procedures but not by the other. Therefore, pathwise,
We introduce the notations
\begin{equation}\label{eq:def-mu-min-max}
\underline{\mu}_i(z^-) := \mu_i^K(Q^K(z^-))\wedge \mu^{MF1}(z^-),
\qquad
\overline{\mu}_i(z^-) := \mu_i^K(Q^K(z^-))\vee \mu^{MF1}(z^-).
\end{equation}
With these notations, we obtain the pathwise control
\begin{align}\label{eq:pathwise-control-one}
\sup_{0\leq s\leq t}
\left|q_i^{(K,1)}(s)-q_i^{(MF1,1)}(s)\right|
&\leq
\int_0^t\!\!\int_{\mathbb R_+}
\mathbf 1_{\left\{
\underline{\mu}_i(z^-)
<
u
\leq
\overline{\mu}_i(z^-)
\right\}}
\,D_{1,i}(\mathrm du,\mathrm dz).
\end{align}

Taking expectations and using the compensation formula for Poisson point
measures yields
\begin{align}\label{eq:first-bound-one}
\mathbb E\left[
\sup_{0\leq s\leq t}
\left|q_i^{(K,1)}(s)-q_i^{(MF1,1)}(s)\right|
\right]
&\leq
\int_0^t
\mathbb E\left[
\left|\mu_i^K(Q^K(z))-\mu^{MF1}(z)\right|
\right]
\,\mathrm dz.
\end{align}
By the Lipschitz continuity of \(g\),
\begin{align}\label{eq:rate-difference}
\mathbb E\left[
\left|\mu_i^K(Q^K(z))-\mu^{MF1}(z)\right|
\right]
&\leq
L\,
\mathbb E\Bigg[
\Bigg|
\frac{1}{K}
\sum_{r=1}^K
\sum_{j\neq i}
\mathbf 1_{\{q_j^{(K,r)}(z)>0\}}
l(d(i,j))
-
\mathbb E\left[I_i^{(MF1,1)}(z)\right]
\Bigg|
\Bigg].
\end{align}
We decompose the right-hand side as
\begin{align*}
&\mathbb E\Bigg[
\Bigg|
\frac{1}{K}
\sum_{r=1}^K
\sum_{j\neq i}
\mathbf 1_{\{q_j^{(K,r)}(z)>0\}}
l(d(i,j))
-
\mathbb E\left[I_i^{(MF1,1)}(z)\right]
\Bigg|
\Bigg]
\\
&\leq
\mathbb E\Bigg[
\Bigg|
\frac{1}{K}
\sum_{r=1}^K
\sum_{j\neq i}
\left(
\mathbf 1_{\{q_j^{(K,r)}(z)>0\}}
-
\mathbf 1_{\{q_j^{(MF1,r)}(z)>0\}}
\right)
l(d(i,j))
\Bigg|
\Bigg]
\\
&\quad+
\mathbb E\left[
\left|
\frac{1}{K}
\sum_{r=1}^K
I_i^{(MF1,r)}(z)
-
\mathbb E\left[I_i^{(MF1,1)}(z)\right]
\right|
\right].
\end{align*}
For the first term, since \(q_j^{(K,r)}(z)\) and \(q_j^{(MF1,r)}(z)\) are
integer-valued,
\[
\left|
\mathbf 1_{\{q_j^{(K,r)}(z)>0\}}
-
\mathbf 1_{\{q_j^{(MF1,r)}(z)>0\}}
\right|
\leq
\left|
q_j^{(K,r)}(z)-q_j^{(MF1,r)}(z)
\right|.
\]
Therefore, using equality in law in \(j\) and \(r\),
\begin{align}\label{eq:first-term-bound}
&\mathbb E\Bigg[
\Bigg|
\frac{1}{K}
\sum_{r=1}^K
\sum_{j\neq i}
\left(
\mathbf 1_{\{q_j^{(K,r)}(z)>0\}}
-
\mathbf 1_{\{q_j^{(MF1,r)}(z)>0\}}
\right)
l(d(i,j))
\Bigg|
\Bigg]
\nonumber\\
&\leq
\xi_N
\mathbb E\left[
\left|
q_i^{(K,1)}(z)-q_i^{(MF1,1)}(z)
\right|
\right].
\end{align}
For the second term, the random variables \(I_i^{(MF1,r)}(z)\) are independent
with respect to \(r\), identically distributed, and bounded by \(\xi_N\). Hence,
by Cauchy--Schwarz,
\begin{align}\label{eq:fluctuation-bound}
\mathbb E\left[
\left|
\frac{1}{K}
\sum_{r=1}^K
I_i^{(MF1,r)}(z)
-
\mathbb E\left[I_i^{(MF1,1)}(z)\right]
\right|
\right]
&\leq
\left(
\operatorname{Var}
\left(
\frac{1}{K}
\sum_{r=1}^K
I_i^{(MF1,r)}(z)
\right)
\right)^{1/2}
\nonumber\\
&=
\frac{1}{\sqrt K}
\left(
\operatorname{Var}
\left(I_i^{(MF1,1)}(z)\right)
\right)^{1/2}
\nonumber\\
&\leq
\frac{\xi_N}{\sqrt K}.
\end{align}
Combining \eqref{eq:first-bound-one}, \eqref{eq:rate-difference},
\eqref{eq:first-term-bound}, and \eqref{eq:fluctuation-bound}, we obtain
\[
\mathbb E\left[
\sup_{0\leq s\leq t}
\left|q_i^{(K,1)}(s)-q_i^{(MF1,1)}(s)\right|
\right]
\leq
L\xi_N
\int_0^t
\mathbb E\left[
\left|q_i^{(K,1)}(z)-q_i^{(MF1,1)}(z)\right|
\right]
\,\mathrm dz
+
\frac{L\xi_N t}{\sqrt K}.
\]
Define
\[
\delta_K(t)
:=
\mathbb E\left[
\sup_{0\leq s\leq t}
\left|q_i^{(K,1)}(s)-q_i^{(MF1,1)}(s)\right|
\right].
\]
Since
\[
\mathbb E\left[
\left|q_i^{(K,1)}(z)-q_i^{(MF1,1)}(z)\right|
\right]
\leq
\delta_K(z),
\]
we get
\[
\delta_K(t)
\leq
L\xi_N
\int_0^t \delta_K(z)\,\mathrm dz
+
\frac{L\xi_N t}{\sqrt K}.
\]
By Grönwall's lemma,
\[
\delta_K(T)
\leq
\frac{L\xi_N T}{\sqrt K}
\exp(L\xi_N T).
\]
In particular,
\[
\mathbb E\left[
\sup_{0\leq t\leq T}
\left|q_i^{(K,1)}(t)-q_i^{(MF1,1)}(t)\right|
\right]
\xrightarrow[K\to\infty]{}0.
\]

Let \(d_{\mathrm{Sk}}\) denote the usual Skorokhod distance on
\(D([0,T],\mathbb N)\), and define its truncated version by
\[
\bar d_{\mathrm{Sk}}(x,y)
:=
d_{\mathrm{Sk}}(x,y)\wedge 1.
\]
Since the Skorokhod distance is bounded from above by the uniform distance,
we have, for all \(x,y\in D([0,T],\mathbb N)\),
\[
\bar d_{\mathrm{Sk}}(x,y)
\leq
d_{\mathrm{Sk}}(x,y)
\leq
\sup_{0\leq t\leq T}|x(t)-y(t)|.
\]
We denote the two coupled paths by
\begin{equation}\label{eq:def-paths}
Q_i^{K} := \left(q_i^{(K,1)}(t)\right)_{0\leq t\leq T},
\qquad
Q_i^{MF1} := \left(q_i^{(MF1,1)}(t)\right)_{0\leq t\leq T}.
\end{equation}
Hence, using the coupling constructed above,
\[
W_1^{\bar d_{\mathrm{Sk}}}
\left(
\mathcal L\left(Q_i^{K}\right),
\mathcal L\left(Q_i^{MF1}\right)
\right)
\leq
\mathbb E\left[
\bar d_{\mathrm{Sk}}
\left(
Q_i^{K},
Q_i^{MF1}
\right)
\right].
\]

Therefore,
\[
W_1^{\bar d_{\mathrm{Sk}}}
\left(
\mathcal L\left(Q_i^{K}\right),
\mathcal L\left(Q_i^{MF1}\right)
\right)
\leq
\delta_K(T)
\leq
\frac{L\xi_N T}{\sqrt K}
\exp(L\xi_N T).
\]
Consequently,
\[
W_1^{\bar d_{\mathrm{Sk}}}
\left(
\mathcal L\left(Q_i^{K}\right),
\mathcal L\left(Q_i^{MF1}\right)
\right)
\xrightarrow[K\to\infty]{}0.
\]

Since the truncated distance \(\bar d_{\mathrm{Sk}}\) is bounded and induces the
same topology as the usual Skorokhod distance \(d_{\mathrm{Sk}}\), convergence
in the Wasserstein distance \(W_1^{\bar d_{\mathrm{Sk}}}\) is equivalent to weak
convergence of the laws on \(D([0,T],\mathbb N)\) (see \cite{villani2008optimal}). Therefore,
\[
\left(q_i^{(K,1)}(t)\right)_{0\leq t\leq T}
\xrightarrow[K\to\infty]{d}
\left(q_i^{(MF1,1)}(t)\right)_{0\leq t\leq T}
\]
in \(D([0,T],\mathbb N)\) endowed with the Skorokhod topology.
\end{proof}

\begin{proof}[Proof of \Cref{thm:chaos-joint}]
We couple the finite replica system and the McKean--Vlasov system as in
\Cref{thm:chaos-one}, using the same initial conditions, arrival processes
and Poisson point measures.

For every \(i\in\{1,\dots,N\}\), \Cref{thm:chaos-one} gives
\[
\mathbb E\left[
\sup_{0\leq t\leq T}
\left|q_i^{(K,1)}(t)-q_i^{(MF1,1)}(t)\right|
\right]
\leq
\frac{L\xi_N T}{\sqrt K}
\exp(L\xi_N T).
\]
Hence,
\[
\begin{aligned}
\mathbb E\left[
\sup_{1\leq i\leq N}
\sup_{0\leq t\leq T}
\left|q_i^{(K,1)}(t)-q_i^{(MF1,1)}(t)\right|
\right]
&\leq
\sum_{i=1}^N
\mathbb E\left[
\sup_{0\leq t\leq T}
\left|q_i^{(K,1)}(t)-q_i^{(MF1,1)}(t)\right|
\right]
\\
&\leq
\frac{N L\xi_N T}{\sqrt K}
\exp(L\xi_N T).
\end{aligned}
\]
This converges to \(0\) as \(K\to\infty\).

Therefore, by the same argument as in \Cref{thm:chaos-one}, using the
Wasserstein distance associated with the truncated product Skorokhod distance,
we obtain
\[
\left(
\left(q_i^{(K,1)}(t)\right)_{0\leq t\leq T}
\right)_{i=1}^N
\xrightarrow[K\to\infty]{d}
\left(
\left(q_i^{(MF1,1)}(t)\right)_{0\leq t\leq T}
\right)_{i=1}^N
\]
in \(D([0,T],\mathbb N)^N\), endowed with the product Skorokhod topology.

It remains to identify the law of the limiting vector. In the McKean--Vlasov
system, the rate \(\mu^{MF1}(t)\) is deterministic. Moreover, the arrival
processes \(A_{1,i}\), the Poisson point measures \(D_{1,i}\), and the initial
conditions \(q_i^{(MF1,1)}(0)\) are independent with respect to \(i\). Hence the
processes
\[
\left(q_i^{(MF1,1)}(t)\right)_{0\leq t\leq T},
\qquad i=1,\dots,N,
\]
are independent, and their joint law is
\[
\bigotimes_{i=1}^N
\mathcal L\left(
\left(q_i^{(MF1,1)}(t)\right)_{0\leq t\leq T}
\right).
\]
This proves the theorem.
\end{proof}

\begin{proof}[Proof of \Cref{thm:mf1-stationary}]
    We proceed by an analytic–synthetic proof. Suppose that the solution of \eqref{eq:MF1} admits a stationary distribution and initialise from it. Then the solution of \eqref{eq:MF1} is also the solution of an M/M/1 queue.  Being a solution of an M/M/1 queue and being in stationary regime impose that the stationary distribution is  $\pi_\lambda$ parameterized by the unique solution of \eqref{eq:mf1-fixedpoint}. Consider now that the system is initialised by $\pi_\lambda$ for the corresponding M/M/1 queue. By the balance equation, it is in stationary regime. Moreover it satisfies \eqref{eq:MF1} hence it is a stationary measure for \eqref{eq:MF1}.
\end{proof}

We now derived the proof concerning the second McKean-Vlasov approximation. 

\begin{proof}[Proof of \Cref{thm:mf2-fixedpoint}]
Assume that $\lambda < \mu^*$,

Fix \(p\in[0,1]\). In stationarity, the rate conservation equation gives
\[
\lambda
=
\mathbb E\left[
\mu_i^p(\Theta)\mathbf 1_{\{q_i^p>0\}}
\right].
\]
By construction, the environment \(\Theta\) used at a potential service epoch
is sampled independently of the current queue state. Therefore,
\[
\mathbb E\left[
\mu_i^p(\Theta)\mathbf 1_{\{q_i^p>0\}}
\right]
=
\mathbb E_{\nu_p}\left[\mu_i^p(\Theta)\right]
\mathbb P(q_i^p>0).
\]
Hence
\[
\mathbb P(q_i^p>0)
=
\frac{\lambda}{\overline\mu_i(p)}.
\]

Define
\[
G(p):=\frac{\lambda}{\overline\mu_i(p)}-p.
\]
We now show that \(G\) has a zero on \([0,1]\).

The map \(p\mapsto \overline\mu_i(p)\) is continuous. Indeed,
\[
\overline\mu_i(p)
=
\sum_{\theta\in\{0,1\}^{N-1}}
g\left(
\sum_{m\neq i}\theta_m l(d(i,m))
\right)
p^{|\theta|}(1-p)^{N-1-|\theta|},
\]
where
\[
|\theta|=\sum_{m\neq i}\theta_m.
\]
Thus \(\overline\mu_i(p)\) is a finite sum of continuous functions of \(p\).
Moreover, \(\overline\mu_i(p)>0\) for all \(p\in[0,1]\), so \(G\) is continuous.

At \(p=0\), we have
\[
G(0)
=
\frac{\lambda}{\overline\mu_i(0)}
>0.
\]
At \(p=1\), all potential interferers are active, hence
\[
\overline\mu_i(1)=\mu^*.
\]
Therefore,
\[
G(1)
=
\frac{\lambda}{\mu^*}-1
<0,
\]
because \(\lambda<\mu^*\).

By the intermediate value theorem, there exists \(p^\star\in(0,1)\) such that
\[
G(p^\star)=0.
\]
Equivalently,
\[
p^\star
=
\frac{\lambda}{\overline\mu_i(p^\star)}
=
\mathbb P(q_i^{p^\star}>0).
\]
This proves the result.
\end{proof}
\begin{proof}[Proof of \Cref{prop:mf2-mf1}]
Take $\lambda<\mu^*$. By the balance equation we have that $$\mathbb{P}(q^{MF2}>0)=\frac{\lambda}{\mathbb{E}_{\nu_p}[\mu]}.$$ 
Now remark that $g:x\mapsto\log(1+\frac{1}{1+x})$ is convex since $g''>0$. Hence Jensen's inequality gives us : $$\mathbb{E}\!\left[\mu^{\mathrm{MF2}}\right]
\geq
B \log_2\!\left(
1 + \frac{1}{1 + \xi_N \, \mathbb{P}\!\left(q^{\mathrm{MF2}} > 0\right)}
\right).$$ 
Combining the last two equations we end up with: $$\mathbb{P}(q^{MF2}>0)\leq \frac{\lambda}{B\log_2(1+\frac{1}{1+\xi_N\mathbb{P}(q^{MF2}>0)})}.$$ Now using the analysis we did in \Cref{lem:mf1-fixedpoint} with $f:=x\mapsto \frac{\lambda}{x}-B\log(1+\frac{1}{1+\xi_Nx})$, we have $f(\mathbb{P}(q^{MF2}>0))\geq0$ and because of the property of $f$, it implies that \begin{align}
    \mathbb{P}(q^{MF2}>0)&\leq x^* \notag\\
    \mathbb{P}(q^{MF2}>0)&\leq\pi^{MF1}_\lambda(q>0)
\end{align}
    
\end{proof}

\subsection{Long Time behavior}

We prove a Foster--Lyapunov drift condition for the original system. We then use
the exponential ergodicity criterion of \cite{MeynTweedie1993}. In the present
countable-state setting, the petite set condition is automatically satisfied
for finite sets. Therefore, once a drift condition of the form
\[
\mathcal L V \leq -cV + b\mathbf 1_C
\]
has been established with \(C\) finite, the process is \(V\)-uniformly
exponentially ergodic.

We denote by \(\mathbb P_x^{FB}\) the law of the full-buffer process started
from \(x\), namely the process consisting of \(N\) independent
\(M/M/1(\lambda,\mu^*)\) queues. Its generator is denoted by
\(\mathcal L^{FB}\).

The Lyapunov function introduced below was obtained with the assistance of an AI tool (ChatGPT Thinking 5.2 ). 

\begin{lemma} \label{lem:lyapunov}
Assume that \(\lambda<\mu^*\).
There exists \(r>1\) such that the function
\[
V(x)=\sum_{i=1}^N r^{x_i}, \qquad x\in\mathbb N^N,
\]
satisfies the Foster--Lyapunov drift condition
\[
\mathcal LV(x)\leq \mathcal L^{FB}V(x)\leq -cV(x)+b\mathbf 1_C(x),
\]
for some constants \(c,b>0\) and some finite set \(C\subset \mathbb N^N\).
\end{lemma}

\begin{remark} \label{rem:extended-generator}
The function \(V\) is unbounded, whereas the operator \(\mathcal L^{FB} \)was initially
defined on bounded test functions. Throughout this argument, \(\mathcal L^{FB} \) is
understood as the extended generator. For the present countable-state,
non-explosive jump process, with bounded jump, the extended generator coincides pointwise with the
formal jump operator. In particular, every pointwise finite function \(f\) for
which \(\mathcal L^{FB}  f(x)\) is finite for all \(x\) belongs to the domain of the
extended generator; see \cite{MEYN1993518}.
\end{remark}

\begin{proof}{Proof of \Cref{thm:exponential-ergodicity}}
    By the previous lemma, the process satisfies a Foster--Lyapunov drift condition
of the form
\[
\mathcal LV \leq -cV+b\mathbf 1_C,
\]
where \(C\) is finite. Since the process is irreducible on a countable state
space, every finite set is petite. Therefore the exponential ergodicity theorem
of \cite{MeynTweedie1993} applies and yields constants \(M<\infty\) and
\(\kappa>0\) such that
\[
\|\delta_x P_t-\pi_\lambda\|_{TV}
\leq M V(x)e^{-\kappa t},
\qquad t\geq 0.
\]
The result follows by taking \(C_x=MV(x)\).
\end{proof}

\subsection{Random positions}

\begin{lemma} \label{lem:low-power-root}
    For any $\epsilon>0$, for any $\delta \in (0;1)$, $\exists P_e$  such that for $ 0<\lambda<(1-\delta)B\log_2(1+\frac{P_e}{1+\frac{J_N}{a^\alpha}})$,  we have $x^*(s)=\frac{\lambda}{BP_e/\ln(2)-\lambda\frac{J_N}{s^\alpha}}+\epsilon$
\end{lemma}

\begin{proof}
    Do the Taylor expansion of order 1 of $\log(1+x)$ in:$$\frac{B}{\ln(2)}\log(1+\frac{P_e}{1+x^*(s)\frac{J_N}{s^\alpha}})=\frac{\lambda}{x^*(s)}.$$
\end{proof}

\begin{lemma} \label{lem:low-power-mean}
    For any $\epsilon>0$, for any $\delta \in (0;1)$, $\exists P_e$ such that  for $0<\lambda<(1-\delta)B\log_2(1+\frac{P_e}{1+\frac{J_N}{a^\alpha}})$, $$\mathbb{E}^{MF1}[q|S=s]=\frac{\lambda}{BP_e/\ln(2)-\lambda(1+\frac{J_N}{s^\alpha})}+ \epsilon .$$ 
\end{lemma}

\begin{proof}
    Use \Cref{lem:low-power-root} and that $\mathbb{E}^{MF1}[q|S=s]=\frac{x^*(s)}{1-x^*(s)}$.
\end{proof}

\begin{proof}[Proof of \Cref{prop:random-spacing}]
    Use the previous lemma and apply the transfer theorem and the change of variable $u=\frac{J_N}{s^\alpha}. $
\end{proof}
\section{Extensions, variants, and open problems}\label{sec:extensions}

\subsection{Possible generalizations}

As discussed in the introduction, the model considered in this paper is deliberately idealized and relies on several assumptions that are unlikely to hold exactly in practical systems. A first natural direction is to relax the assumption that transmitters are equally spaced. In this case, the queues are no longer statistically identical, which considerably complicates the analysis. At the same time, this setting reveals a richer behaviour, including the possible coexistence of stable and unstable queues. In forthcoming work, we will investigate this phenomenon through the order of explosion and the characterization of the spatial regions in which queues become unstable.

In this more general framework, the mean-field approximation leads to a fixed-point equation in $\mathbb{R}^N$, rather than to a scalar equation. It is also natural to question the assumption of homogeneous arrival rates across transmitters, especially in cellular networks. Since cells may have different sizes, the arrival rate should in general depend, at least partially, on the size of the corresponding cell. This heterogeneous setting will also be addressed in subsequent work.

A further possible approximation consists in retaining only nearest-neighbour interactions and neglecting longer-range interference terms. Such a local-interaction model may provide a tractable intermediate regime between the full interacting system and the mean-field approximations studied here. This direction will be investigated in future work.

A promising direction for future work is to study the impact of decoding interference through successive interference cancellation (SIC), rather than treating it as noise. This could potentially lead to a larger achievable capacity region and provide stronger evidence in favor of full-spectrum sharing strategies. This topic is currently under investigation.

Finally, another important extension is the incorporation of feedback mechanisms. In particular, one may ask whether some system parameters can be dynamically controlled in order to improve performance, for instance by reducing delay, congestion, or instability. This control-oriented perspective is currently under investigation.

\subsection{Open problems}

Several questions remain open and deserve further study:
\begin{itemize}[leftmargin=2em]
    \item Can one derive quantitative error bounds between the true system and the proposed mean-field approximations?
    
    \item Can uniqueness be proved for the fixed-point equation \eqref{eq:fp3}?
    
    \item Can one exhibit a limiting system from which the second mean-field approximation arises rigorously?

    \item Can one manage to study the long time behavior of the first Mean-field approximation ? 
\end{itemize}

\appendix
\section{Appendix}\label{sec:appendix}

This appendix contain the technical results. 

\begin{lemma} \label{lem:sde-construction}
    There exist a path wise unique solution to \eqref{eq:sde}
\end{lemma}
\begin{proof}
    Since $\mu$ is bounded and $\lambda$ is constant conditions of Theorem 117 of \cite{situ2005jumps} applies. 
\end{proof}
\begin{proposition} \label{prop:sde-generator}
    Let $Q$ have generator~\eqref{eq:generator}, and let $\widetilde Q$ satisfy~\eqref{eq:sde}.
    If $Q$ and $\Tilde{Q}$ have the same law for the initial condition then 
   $Q \overset{d}{=} \Tilde{Q}$
\end{proposition}

\begin{proof}
Since the jump rates are uniformly bounded, the martingale problem associated with
\((\mathcal L,\mathcal B_b)\) is well posed; see \cite{ETHIER1986}. It is therefore
enough to show that the process defined by the stochastic equation solves this
martingale problem.

Let $h$ be a bounded function. By the jump representation of \(Q\),
\[
\begin{aligned}
h(Q(t))-h(Q(0))
&=
\sum_{i=1}^N
\int_0^t\int_0^\infty
\Delta_i^+ h(Q(s-))
\mathbf 1_{\{u\leq \lambda\}}\,
M_i^a(\mathrm du,\mathrm ds)
\\
&\quad+
\sum_{i=1}^N
\int_0^t\int_0^\infty
\Delta_i^- h(Q(s-))
\mathbf 1_{\{Q_i(s-)>0\}}
\mathbf 1_{\{u\leq \mu_i(Q(s-))\}}\,
M_i^d(\mathrm du,\mathrm ds),
\end{aligned}
\]
where
\[
\Delta_i^+ h(q):=h(q+e_i)-h(q),
\qquad
\Delta_i^- h(q):=h(q-e_i)-h(q).
\]
Adding and subtracting the compensators gives
\[
h(Q(t))-h(Q(0))
=
M_t^h+\int_0^t \mathcal Lh(Q(s))\,\mathrm ds,
\]
where
\[
\begin{aligned}
M_t^h
&=
\sum_{i=1}^N
\int_0^t\int_0^\infty
\Delta_i^+ h(Q(s-))
\mathbf 1_{\{u\leq \lambda\}}\,
\widetilde M_i^a(\mathrm du,\mathrm ds)
\\
&\quad+
\sum_{i=1}^N
\int_0^t\int_0^\infty
\Delta_i^- h(Q(s-))
\mathbf 1_{\{Q_i(s-)>0\}}
\mathbf 1_{\{u\leq \mu_i(Q(s-))\}}\,
\widetilde M_i^d(\mathrm du,\mathrm ds),
\end{aligned}
\]
with
\[
\widetilde M_i^a(\mathrm du,\mathrm ds)
:=
M_i^a(\mathrm du,\mathrm ds)-\mathrm du\,\mathrm ds,
\qquad
\widetilde M_i^d(\mathrm du,\mathrm ds)
:=
M_i^d(\mathrm du,\mathrm ds)-\mathrm du\,\mathrm ds.
\]
Since \(Q(s)=Q(s-)\) for Lebesgue-a.e. \(s\), the drift term is
\[
\int_0^t \mathcal Lh(Q(s))\,\mathrm ds,
\]
where
\[
\mathcal Lh(q)
=
\sum_{i=1}^N
\lambda \Delta_i^+h(q)
+
\sum_{i=1}^N
\mu_i(q)\mathbf 1_{\{q_i>0\}}\Delta_i^-h(q).
\]
Since \(h\) is bounded and the intensities are uniformly bounded, the integrands are bounded and predictable and the above
stochastic integrals are martingales. Thus, for every
\(h\in\mathcal B_b(\mathbb N^N)\),
\[
h(Q(t))-h(Q(0))-\int_0^t\mathcal Lh(Q(s))\,\mathrm ds
\]
is a martingale. Hence \(Q\) solves the martingale problem for \(\mathcal L\), and by
well-posedness it is a Markov process with generator \(\mathcal L\).
\end{proof}

\begin{lemma}[Monotonicity]\label{lem:mono}
The Markov chain $Q$ is monotone with respect to the arrival rate, the initial condition, and the departure rate. More precisely:
\begin{itemize}
    \item \textbf{Monotonicity in the arrival rate.} If $\lambda_1 < \lambda_2$, then the chains $Q^{(1)}$ and $Q^{(2)}$ defined by \eqref{eq:generator} with arrival rates $\lambda_1$, $\lambda_2$ and $Q^{(1)}(0)=Q^{(2)}(0)$ satisfy
    \[
    Q^{(1)} \leq_{\mathrm{st}} Q^{(2)}.
    \]

    \item \textbf{Monotonicity in the initial condition.} If $x \leq y$, then the chains $Q^x$ and $Q^y$ defined by \eqref{eq:generator} with initial conditions $Q^x(0)=x$ and $Q^y(0)=y$ satisfy
    \[
    Q^x \leq_{\mathrm{st}} Q^y.
    \]

    \item \textbf{Monotonicity in the departure rate.} If for all $i \in [1;N]$, $f \leq \mu_i$ (resp. $f \geq \mu_i$), and $Q^f(0)=Q(0)$, then the chain $Q^f$ defined by \eqref{eq:generator} with departure rate $\mu_i = f$ for every $i$ satisfies
    \[
    Q \leq_{\mathrm{st}} Q^f
    \qquad (\text{resp. } Q^f \leq_{\mathrm{st}} Q).
    \]
\end{itemize}
\end{lemma}

\begin{proof}
Our argument relies on Strassen's theorem (see Theorem 4.2.1 in \cite{baccelli1994elements}), which reduces the proof of a stochastic domination to the construction of a probability space on which the two processes can be ordered almost surely.

Let $\lambda_1 \leq \lambda_2$ and let both processes start from the same
initial condition $x \in \mathbb N^N$. We construct $Q^{\lambda_1}$ and
$Q^{\lambda_2}$ on the same probability space, using the same Poisson
measures $(A_i)$ and $(D_i)$. For $k=1,2$, the process $Q^{\lambda_k}$ is
defined by
\begin{align*}
q_i^{\lambda_k}(t)
&=
x_i
+ \int_0^t \int_{\mathbb R_+}
\mathbf 1_{\{z \leq \lambda_k\}}
\, A_i(\mathrm dz,\mathrm ds)
\\
&\quad
- \int_0^t \int_{\mathbb R_+}
\mathbf 1_{\{q_i^{\lambda_k}(s^-)>0\}}
\mathbf 1_{\{u \leq \mu_i(Q^{\lambda_k}(s^-))\}}
\, D_i(\mathrm du,\mathrm ds),
\qquad i=1,\dots,N .
\end{align*}
By \Cref{prop:sde-generator}, these equations define a valid coupling of $Q^{\lambda_1}$ and $Q^{\lambda_2}$.

Since $\lambda_1 \leq \lambda_2$, every arrival accepted by
$Q^{\lambda_1}$ is also accepted by $Q^{\lambda_2}$. We show that
\[
Q^{\lambda_1}(t) \leq Q^{\lambda_2}(t),
\qquad t \geq 0,
\]
almost surely.

The coupled system has locally finitely many jumps on compact time
intervals. The order holds at time $0$. Assume that it holds on $[0,T)$,
where $T$ is a jump time, and write
\[
Q^{\lambda_1}(T^-) \leq Q^{\lambda_2}(T^-).
\]
We check that the order is preserved at time $T$.

If the jump is an arrival, then either both processes receive the arrival,
or only $Q^{\lambda_2}$ receives it. In both cases the order is preserved.

Now consider a departure attempt on coordinate $i$. The order can only be
violated if
\[
q_i^{\lambda_1}(T^-)=q_i^{\lambda_2}(T^-)
\]
and if the departure occurs in $Q^{\lambda_2}$ but not in
$Q^{\lambda_1}$. If this common value is zero, no departure can occur in
either process because of the factor
\[
\mathbf 1_{\{q_i^{\lambda_k}(T^-)>0\}}.
\]
If the common value is positive, then the departure indicator due to
non-emptiness is equal to one for both processes. Moreover, since
\[
Q^{\lambda_1}(T^-) \leq Q^{\lambda_2}(T^-)
\]
and since $g$ is decreasing, we have
\[
\mu_i(Q^{\lambda_2}(T^-))
\leq
\mu_i(Q^{\lambda_1}(T^-)).
\]
Therefore, any atom satisfying
\[
u \leq \mu_i(Q^{\lambda_2}(T^-))
\]
also satisfies
\[
u \leq \mu_i(Q^{\lambda_1}(T^-)).
\]
Thus a departure of $Q^{\lambda_2}$ on coordinate $i$ necessarily implies
a departure of $Q^{\lambda_1}$ on the same coordinate. Hence the order is
preserved.

By induction over the jump times, we obtain
\[
Q^{\lambda_1}(t) \leq Q^{\lambda_2}(t),
\qquad t \geq 0,
\]
almost surely. This gives
\[
Q^{\lambda_1} \leq_{\mathrm{st}} Q^{\lambda_2}.
\]

The arguments for ordered initial conditions and ordered departure rates follow exactly the same lines
\end{proof}

\begin{proof}[Proof of \Cref{lem:mf1-fixedpoint}]
 We study the function 
\[
f : x \mapsto \frac{\lambda}{x} - B\log_2\left(1+\frac{1}{1+x\xi_N}\right)
\]
on \([0,1]\). Its derivative is given by
\[
f'(x)=\frac{A(x)}{x^2(1+x\xi_N)(2+x\xi_N)},
\]
with
\[
A(x):=\xi_N(\frac{B}{ln(2)}-\lambda\xi_N)x^2 - 3\lambda\xi_N x - 2\lambda.
\]
Note that \(f(0)=+\infty\) and \(f(1)=\lambda-\mu^*\)<0 (since \(\lambda \in [0,\mu^*)\)). Because \(f\) is continuous, the intermediate value theorem guarantees the existence of a solution. For uniqueness, observe that, on the domain where \(f'\) is defined, its sign coincides with that of \(A\). The zeros of \(A\) are determined by the discriminant 
\[
\Delta=(3\lambda\xi_N)^2+8\lambda\xi_N(\frac{B}{ln(2)}-\lambda\xi_N),
\]
which is positive since \(\lambda < \mu^*\) and \(\log(1+x)<x\) for \(x>0\). Moreover, the first root of \(A\) is 
\[
x_1:=\frac{3\lambda\xi_N-\sqrt{\Delta}}{2\xi_N(\frac{B}{ln(2)}-\lambda\xi_N)}<0,
\]
thus \(f\) can cross the axis only once.

\end{proof}

\begin{proof}[Proof of \Cref{lem:lyapunov}]
For the full-buffer process, the generator is given by
\[
(\mathcal L^{FB} f)(x)
=
\sum_{i=1}^N
\left[
\lambda \big(f(x+e_i)-f(x)\big)
+
\mu^* \mathbf 1_{\{x_i>0\}}
\big(f(x-e_i)-f(x)\big)
\right].
\]

Let \(r>1\) and define
\[
V(x)=\sum_{i=1}^N r^{x_i}.
\]
For the original process, we have
\[
\mathcal{L}V(x)
=
\sum_{i=1}^N r^{x_i}
\left[
\lambda(r-1)
+
\mu_i(x)(r^{-1}-1)\mathbf 1_{\{x_i>0\}}
\right].
\]
Since \(r>1\), we have \(r^{-1}-1<0\).
Therefore,
\[
\mu_i(x)(r^{-1}-1)\mathbf 1_{\{x_i>0\}}
\leq
\mu^*(r^{-1}-1)\mathbf 1_{\{x_i>0\}},
\]
and hence
\[
\mathcal{L}V(x)\leq \mathcal{L}^{FB}V(x).
\]

Then: 
\begin{itemize}
\item if $x_i\ge 1$,
\[
\mathcal L^{FB}\big(r^{x_i}\big)(x)
=
\lambda\big(r^{x_i+1}-r^{x_i}\big)
+
\mu^*\big(r^{x_i-1}-r^{x_i}\big)
=
r^{x_i}\Big(\lambda(r-1)+\mu^*(r^{-1}-1)\Big);
\]
\item if $x_i=0$,
\[
\mathcal L^{FB}\big(r^{x_i}\big)(x)
=
\lambda(r^1-r^0)
=
\lambda(r-1).
\]

\[
a(r)
:=
\lambda(r-1)+\mu^*(r^{-1}-1)
=
\lambda r + \frac{\mu^*}{r} - (\lambda+\mu^*).
\]
Because $\lambda<\mu^*$, we can choose $r>1$ such that
 $a(r)<0$.
Indeed
\[
a(1)=0,
\qquad
a'(r)=\lambda-\frac{\mu^*}{r^2}
\ \Rightarrow\
a'(1)=\lambda-\mu^*<0.
\]
Hence,
\[
\mathcal L^{FB} V(x)
=
a(r)\sum_{i:x_i>0} r^{x_i}
+
\lambda(r-1)\,\#\{i:x_i=0\}.
\]

Define
\[
n_0(x):=\#\{i:x_i=0\}\le N.
\]
We have
\[
\sum_{i:x_i>0} r^{x_i} = V(x)-n_0(x),
\]
Hence 
\begin{align}
    \mathcal L^{FB} V(x)
&=
a(r)\big(V(x)-n_0(x)\big)+\lambda(r-1)n_0(x) \nonumber\\
&=
a(r)V(x)+\big(\lambda(r-1)-a(r)\big)n_0(x).
\end{align}

Therefore
\[
\mathcal L^{FB} V(x)
\le
a(r)V(x) + B,
\qquad
B:=N\big(\lambda(r-1)-a(r)\big)<\infty,
\]
with $a(r)<0$.
Define
\[
C:=\{x\in\mathbb{Z}_+^N: V(x)\le M\},
\qquad
M:=\frac{2B}{|a(r)|}.
\]
Then for all $x\notin C$ (i.e. $V(x)>M$),
\begin{align}
\mathcal L^{FB} V(x)
&\le
a(r)V(x)+B \nonumber\\
&=
-|a(r)|\,V(x)+B.\\
&\le
-|a(r)|\,V(x)+\frac{|a(r)|}{2}V(x) \nonumber\\
&=
-\frac{|a(r)|}{2}V(x).
\end{align}
Hence the inequality:
\[
\mathcal L^{FB} V(x)
\le
-c\,V(x)+b\,\mathbf 1_C(x),
\qquad
c:=\frac{|a(r)|}{2},
\qquad
b:=B+cM.
\]

\end{itemize}
\end{proof}

The Jupyter notebook used for the simulations is available in the
\href{https://github.com/Philippe-Sarotte/Interference-Queueing-Networks-A-Replica-Mean-Field-Approach-in-the-Symmetric-Setting}{associated GitHub repository} (https://github.com/Philippe-Sarotte/Interference-Queueing-Networks-A-Replica-Mean-Field-Approach-in-the-Symmetric-Setting).

\bibliographystyle{IEEEtran}
\bibliography{biblio}

\end{document}